\documentclass[11pt]{amsart}

\usepackage{amssymb,amscd,amsthm,amsxtra}
\usepackage{dsfont,latexsym, color}
\usepackage{mathrsfs}

\usepackage{color}
\definecolor{citation}{rgb}{0.2,0.58,0.2} 
\definecolor{formula}{rgb}{0.1,0.2,0.6}
\definecolor{url}{rgb}{0.3,0,0.5} 
\definecolor{marrone}{rgb}{0.7,0.45,0.36} 
 
\usepackage[colorlinks=true,linkcolor=formula,urlcolor=url,citecolor=citation]{hyperref}

\newtheorem{thm}{Theorem}[section]

\newtheorem{lemma}[thm]{Lemma}

\theoremstyle{definition}
\newtheorem{defn}[thm]{Definition}
\theoremstyle{remark}
\newtheorem{rem}[thm]{Remark}
\numberwithin{equation}{section}

\newcommand{\R}{{\mathds R}}
\newcommand{\Om}{{\Omega}}
\newcommand{\Mz}{\mathcal{M}}

\newcommand{\qq}{{\mathcal Q}}
\newcommand{\dx}{\,{\rm d}x}
\newcommand{\dt}{\,{\rm d}t}

\def\Xint#1{\mathchoice
      {\XXint\displaystyle\textstyle{#1}}%
      {\XXint\textstyle\scriptstyle{#1}}%
      {\XXint\scriptstyle\scriptscriptstyle{#1}}%
      {\XXint\scriptscriptstyle\scriptscriptstyle{#1}} %
\!\int}
   \def\XXint#1#2#3{{\setbox0=\hbox{$#1{#2#3}{\int}$}
        \vcenter{\hbox{$#2#3$}}\kern-.5\wd0}}
   
   \def\dashint{\Xint-}

\newlength{\defbaselineskip}
\newcommand{\setlinespacing}[1]
           {\setlength{\baselineskip}{#1 \defbaselineskip}}

\makeatletter
\DeclareRobustCommand*{\bfseries}{%
  \not@math@alphabet\bfseries\mathbf
  \fontseries\bfdefault\selectfont
  \boldmath
}
\makeatother

\title 
[Global estimates for nonlinear parabolic equations]
{Global estimates \\[0.5ex]  for nonlinear parabolic equations}
\author[P. Baroni]{Paolo Baroni}
\email[P. Baroni]{\href{mailto:paolo.baroni@math.uu.se}{paolo.baroni@math.uu.se}}
\author[A. Di Castro]{Agnese Di Castro}
\email[A. Di Castro]{\href{mailto:agnese.dicastro@unipr.it}{agnese.dicastro@unipr.it}}
\author[G. Palatucci]{Giampiero Palatucci}
\email[G. Palatucci]{\href{mailto:giampiero.palatucci@unimes.fr}{giampiero.palatucci@unimes.fr}}

\address[P. Baroni]
{Deparment of Mathematics, Uppsala Universitet\\
L\"agerhyddsv\"agen 1, Box 480,
S-75106 Uppsala, Sweden}

\address[A. Di Castro, G. Palatucci]
{Dipartimento di Matematica, Universit\`a degli Studi di Parma
\\ Campus - Parco Area delle Scienze,~53/A
\\ 43124 Parma, Italia}

\begin{document}

\subjclass[2010]{Primary 35K55, 35B65;
Secondary 35K10, 46E30\vspace{1mm}}

\keywords{Nonlinear parabolic problems, Calder\'on-Zygmund theory, Measure data, Lorentz regularity, 
Rearrangement function spaces, Higher integrability.
\vspace{1mm}}

\thanks{{\em Thanks.} This work has been supported by the \href{http://prmat.math.unipr.it/~rivista/eventi/2010/ERC-VP/}{ERC grant 207573 ``Vectorial Problems''}. 
The authors would like to thank Giuseppe Mingione for his constant encouragement in their research.
}

\begin{abstract}
\small 
We consider nonlinear parabolic equations of the type
$$
u_t - \text{div}\, a(x, t, Du)= f(x,t) \ \,  \text{on} \ \Omega_T = \Omega\times (-T,0),
$$
under standard growth conditions on $a$, with $f$ only assumed to be integrable. 
We prove general decay estimates up to the boundary for level sets of the solutions $u$ and the gradient $Du$ which imply very general estimates in Lebesgue and Lorentz spaces. Assuming only that the  involved domains satisfy a mild exterior capacity density condition, we provide global regularity results.

\end{abstract}

\begin{center}
\begin{large}
{\sc To appear in \ J. Evol. Equations}
\end{large}
\end{center}
\vspace{1cm}

\maketitle

\setcounter{tocdepth}{2}

{\small \tableofcontents}

\setlinespacing{1.04} 

\newpage

\section{Introduction}\label{sec_intro} 
 
The aim of the present paper is to establish {\it global} regularity results for the solutions of the following class of Cauchy-Dirichlet problem
\begin{equation}\label{problema}
\begin{cases}
\displaystyle
u_t-\text{\rm div}\, a(x,t, Du) = f(x,t) \quad& \text{in} \ \Omega_T = \Omega\times (-T,0),\\
u = 0 & \text{on} \ \partial_{{\rm par}}\Om_T,
\end{cases}
\end{equation}
where $\Om$ is a bounded open set in $\R^n$, $n\geq 2$, $T>0$, $\partial_{{\rm par}}\Om_T$ is the usual parabolic boundary of $\Om_T$.
Here we assume that $f$, in the most general case, is a summable function in $\Om_T$, and $a:\Om_T\times\R^n\to\R^n$ is a Carath\'eodory regular vector field satisfying standard monotonicity and growth conditions; i.~\!e., 
\begin{equation}\label{hp_a}
\begin{cases}
\nu |z_2-z_1|^2\leq\langle a(x,t,z_2)-a(x,t,z_1),z_2-z_1 \rangle 
\\[1ex]
|a(x,t,z)|\leq L|z|
\end{cases}
\end{equation}
for every $z_1$, $z_2$, $z\in \mathbb{R}^n$ and $(x,t)\in\Omega_T$; the structure constants  satisfy $0<\nu\leq 1\leq L$.

\vspace{1mm}

We will focus here on the case when $f$ belongs to the Lebesgue space $L^\gamma(\Om_T)$ in a range of $\gamma$ that does not necessarily permit to obtain the existence 
of finite energy solutions $u \in L^2(-T,0;W^{1,2}_0(\Om))$ to problem~\eqref{problema}. 

\begin{rem} \label{remro} In this paper we are not going to consider any solution to the problem \eqref{problema}, but rather we shall deal with the (very) weak solutions  $u\in L^1(-T,0;W^{1,1}_0(\Om))$ obtained via the Boccardo-Gallou\"et by now standard approximation procedure~(\cite{BG89};
see forthcoming Section~\ref{sec_solvability}). These solutions are called SOLA (Solutions Obtained by Limits of Approximations) and turn out to be unique in their class in the case $f \in L^1(\Omega_T)$, see \cite[Theorem 4.1]{DallA}. 
For such reasons, with some abuse of terminology, we shall call 
in the following the unique SOLA to \eqref{problema} ``the solution" to problem \eqref{problema}.
\end{rem}

\vspace{1mm}

In the classical Lebesgue framework, Boccardo, Dall'Aglio, Gallou\"et and Orsina (see \cite[Theorem 1.9]{BDGO97}) show the existence of a unique solution 
\begin{equation}\label{bdgo}
\displaystyle
u \in L^q(-T,0; W^{1,q}_0(\Om)) \ \text{with} \ q=\frac{N\gamma}{N-\gamma}
\end{equation}
to problem~\eqref{problema} under the assumption~\eqref{hp_a}, provided that the datum $f$ satisfies
$$
f\in L^\gamma(\Om_T) \ \text{for some} \ 1 < \gamma < \frac{2N}{N+2}.
$$
Here $N:=n+2$ denotes the homogeneous parabolic dimension. Moreover, the solution $u$ belongs to $L^\sigma(\Om_T)$ with $\sigma=N\gamma/(N-2\gamma)$. This result is optimal in the scale of Lebesgue spaces;  however it is natural to ask for a more accurate scale to describe the regularity of $Du$ in dependence on the datum $f$. 
In this respect, recently, the results above were   
extended in all the most familiar spaces of rearrangement (Lebesgue, Lorentz, Orlicz) and non-rearrangement once, up to Lorentz-Morrey\footnote{We refer to Section~\ref{sec_spaces} for the definitions of the involved function spaces in the parabolic setting.},
by Habermann and the first author of this paper in~\cite{BH12b} (see, also, \cite{KM11,DP12b,KM12b}). 
We would like to remark that a fine analysis in Lorentz spaces regularity is often decisive to establish higher regularity, as it has been shown, for instance, by recent developments in the theory of harmonic and biharmonic mappings (see, e.g.,~\cite{Riv07,HW08,Min12c} and the references therein).
Precisely, in~\cite[Theorem 6.6]{BH12b}, starting from the datum $f$ in the Lorentz space $L(\gamma,q)(\Om_T)$, for $1<\gamma\leq 2N/(N+2)$, the authors prove that for any $0<q\leq\infty$
\begin{equation*}
\displaystyle
|Du| \in L\left(\frac{N\gamma}{N-\gamma}, q\right) \ \text{locally in} \ \Om_T,
\end{equation*}
extending to the parabolic setting a very general approach introduced by Mingione in~\cite{Min07,Min10,Min11} for the regularity of solutions to the elliptic counterpart of problem~\eqref{problema}.

After~\cite{Min07}, many important extensions and generalizations have been considered in the literature, in which a similar approach based on maximal operators techniques has been pursued to achieve regularity results in nonlinear degenerate problems: see, e.~\!g., \cite{Min11,DP12,DP12c} for elliptic problems, \cite{AM07,BH12b,DP12b,DM11,KM12b,KM12c} for parabolic ones and the references therein. It  is important to stress that  the aforementioned papers, despite considering more and more general settings, provide {\it local} results not considering the regularity properties {\it up to the boundary} of the solutions. 
 
 Recently, some attention has been devoted to such global extension, as in the paper by Phuc~\cite{Phu12}, where Lorentz regularity in the whole domain~$\Om$ has been analyzed for the solutions to the elliptic counterpart of problem~\eqref{problema} for general structures of $p$-Laplacian type, again taking into account the original potential approach in~\cite{Min10}\footnote{Although -- as stated in~\cite{Phu12}, too -- the possibility of extending such local results to global ones was already mentioned without proof in~\cite{Min10} and~\cite{Min11b}.
}.
\vspace{1.5mm}

In this paper, we will present a series of up-to-the-boundary regularity results aimed at giving, on one hand, an extension of the classical Lebesgue theory developed in~\cite{BDGO97} to a very general class of rearrangement invariant function spaces. For instance we will cover the case of Lorentz spaces, that do not seem to be treatable with the currently available methods. On the other hand,  we will give a set of techniques allowing for treating domains with very poorly regular boundaries. In this respect we also extend to the parabolic case the recent results obtained in~\cite{Phu12}.

Indeed, since we are dealing with precise estimates up to the boundary, it is necessary to take some regularity condition for $\Om_T=\Om\times(-T,0)$. For this, we assume that the domain $\Om$ has only to satisfy a capacity density condition. Namely, in order to achieve the main results of the present paper, the complement of the set~$\Om$ is required to be a {\it thick domain} (see Section~\ref{sec_thick} below); and such assumption is both very natural (as, for instance, in the classic general Sobolev embedding in capacity density settings; and, also, in the elliptic framework studied in~\cite{Phu12}), and very mild (it is fulfilled by any domain with Lipschitz boundary, or even of corkscrew type; also, in $\R^2$ by any proper simply connected domain). As a further matter, this condition is essentially sharp for higher integrability results, as pointed out  by Kilpel\"ainen and Koskela in Remark~3.3 in the paper \cite{KK94}, where the authors developed a first analysis in global Lebesgue regularity of the gradient of solutions to degenerate elliptic equations in general  
thick domains.

\vspace{1.5mm}

All in all, 
we extend the potential theoretic strategy recently developed by Mingione in~\cite{Min10,Min11} to the parabolic framework up to the boundary, combining it with global higher integrability results (in Theorems~\ref{gehring_boundary} in Section~\ref{sec_comparison}), to obtain general gradient estimates on level sets (see Section~\ref{sec_proofs}), which in turn imply the following
\begin{thm}\label{thm_du}
Let $\Om$ be such that $\R^n\setminus\Om$ is a thick domain with constants $c_0$ and $\rho_0$ and assume~\eqref{hp_a}. Let $f\in L(\gamma,q)(\Omega_T)$ with 
\[
1< \gamma \leq \frac{2N}{N+2}\qquad \text{and}\qquad 0<q\leq \infty.
\]
Then the solution $u\in L^1(-T,0;W^{1,1}_0(\Om))$ to~\eqref{problema} satisfies
$$
|Du| \in L\!\left(\frac{N\gamma}{N-\gamma},q\right) \, \text{in} \ \Omega_T.
$$
Moreover, the following estimate holds
\begin{equation}\label{zina}
\displaystyle
\| Du \|_{L\left(\frac{N\gamma}{N-\gamma},q\right)(\Omega_T)}
\, \leq  \, 
 c\, \| f \|_{L(\gamma, q)(\Omega_T)},
\end{equation}
where $c$ is a constant depending only on $n, q, L, \nu, \gamma, c_0, \rho_0, \text{\rm diam}(\Om)$.
\end{thm}
Of course, by choosing $q=\infty$, we can also deduce regularity results on the Marcinkiewicz scale.

\vspace{1mm}
\begin{rem} After stating Theorem \ref{thm_du}, we want to make a more precise comparison between our results and 
the basic results in~\cite{BDGO97}, where the authors treat the case of Lebesgue spaces; i.~\!e., they take $q=\gamma$  
in the statement above thereby obtaining the corresponding conclusion. The results in Theorem \ref{thm_du} is not a direct extension of the ones in~\cite{BDGO97}, as in this paper no regularity assumptions are considered on $\Omega$. On the other hand, the results in the present paper are global but they are obtained by covering arguments via estimates that are {\em local in nature} as it is easy to see that, using the methods developed here, one can obtain local estimates up to the boundary and not only in the interior.
This corresponds to replace, in Theorem \ref{thm_du}, $\Omega_T$ by cylinders of the type $(B(x_0,R)\cap \Omega) \times (-T, 0)$, where 
$x_0 \in \partial \Omega$. This is a valued consequence of the techniques approached here
and the necessity of assuming the considered boundary capacity conditions 
 stems exactly from this aspect of the involved techniques. Now, whereas it is likely that some Lorentz space estimates can  be also obtained under milder assumption on $\Omega$, it is worth mentioning  that the methods here can  be 
also used to get global estimates in Morrey and Lorentz-Morrey spaces (as those obtained in the interior case in \cite{DP12b,Min10}). 
\end{rem}

The proof of Theorem~\ref{thm_du} 
relies on the fact that the integrability of the spatial gradient of the solutions to problem~\eqref{problema} is linked to a suitable choice of potential operators. In this sense, one of the key-points will consist into obtaining a decay estimate that involves the level sets of the parabolic Hardy-Littlewood maximal operator $M^*_0$ of $|Du|$ in term of those of another maximal operator $M^*_1$ of the datum $f$, up to a correction term which is negligible when considering the higher regularity. Indeed, we will prove an estimate of the following type
\begin{equation}\label{eq_level}
\displaystyle
\big|\big\{ M^*_0(|Du|) \geq S\lambda\big\}\big| \, \lesssim \, \frac{1}{S^{2\chi_1}}\big|\big\{ M_0^*(|Du|)\geq \lambda\big\}\big| 
+ \big|\big\{[M^*_1(|f|)]\geq \lambda\big\}\big|,
\end{equation}
for every $\lambda$ suitably large, and in which $S>\!>1$ is a constant to be chosen, and the exponent $\chi_1>1$ is related to the higher integrability theory.

The estimate in~\eqref{eq_level}, whose precise version is given by forthcoming formula \eqref{tesi.prop.fin}, is fairly general and it will be relevant to deducing the Lorentz estimates stated in our theorems, also including the borderline case.
In order to obtain such level sets estimate, we apply a parabolic covering lemma of Calder\'on-Zygmund type (see Lemma~\ref{lem_covering} below) together with the higher integrability theory which we will extend up to the boundary (see Theorem~\ref{gehring_boundary} below).
Therefore, we will work  on basic estimates of the solutions~$u$ to~\eqref{problema} in comparison to the solutions~$w$ to the corresponding homogeneous problem (see Section~\ref{sec_comparison}), by exploiting very recent contributions in the parabolic framework given in~\cite{KM12c}.
Clearly, the situation is complicated by the necessity to work in a global setting.

\vspace{1.5mm}

Analogously, we will show how to establish a modified version of the level sets estimate in~\eqref{eq_level} by replacing $Du$ by $u$ there. This will permit us to recover higher regularity estimates in the Lorentz scale for the solutions $u$ themselves, as stated in the following
\begin{thm}\label{thm_u}
Let $\Om$ be such that $\R^n\setminus\Om$ is a thick domain with constants $c_0$ and $\rho_0$ and let $f\in L(\gamma,q)(\Omega_T)$ with 
\[
1< \gamma < \frac N2\qquad\text{and}\qquad 0<q\leq\infty.
\]
Then the solution $u\in L^1(-T,0;W^{1,1}_0(\Om))$ to \eqref{problema} satisfies
$$
u \in L\!\left({\frac{N\gamma}{N-2\gamma}},q\right) \, \text{in}\ \Omega_T.
$$
Moreover, the following estimate holds
\begin{equation*}
\displaystyle
\| u \|_{L\left({\frac{N\gamma}{N-2\gamma}},q\right)(\Omega_T)}
\, \leq  \, 
 c\, \| f \|_{L(\gamma,q)(\Omega_T)},
\end{equation*}
where $c$ is a constant depending only on $n, q, L, \nu, \gamma, c_0, \rho_0, \text{\rm diam}(\Om)$.
\end{thm}
It is worth pointing out that, although in the elliptic case the regularity of $u$ can be recovered by plainly combining the regularity of the gradient $Du$ with the classic Sobolev embeddings, here we need to work in a separate way, by means of sharp estimates also involving different maximal operators in~\eqref{eq_level}, and higher regularity results up to the boundary for the solutions $u$. 

\vspace{3mm}

Finally, we recall that all the results we obtained in the present paper hold for the very weak solutions given by the approximation method described in forthcoming Section~\ref{sec_solvability}.
It would be interesting to understand whether these results can be extended to some other notion of solutions. In this respect, a positive answer can be given when dealing with notions of solutions to measure data problems holding uniqueness in case of integrable data, as in case of the renormalized solutions in~\cite{Pet08} and~\cite{PPP11}; 
see~\cite{BM97} for a first definition of renormalized solutions in this framework, and also~\cite{DPP03}. 

\begin{rem} In the present paper we are dealing with the case of operators with linear growth while in other papers, 
as for instance \cite{Min10}, level sets techniques 
 are used to approach similar regularity results for $p$-Laplacian type operators, that is for  equations of the type
$
-{\rm div}\, (|Du|^{p-2}Du)=f.
$
When considering the parabolic analog
$$
u_t-{\rm div}\, (|Du|^{p-2}Du)=f,
$$
the approach of the present paper is not sufficient if used alone, and has to be modified and extended using the concept of intrinsic geometry first pioneered by DiBenedetto to overcome the lack of homogeneity of the evolutionary $p$-Laplacian operator when $p\neq2$ (see for instance~\cite{AM07} for a discussion of this concept in the relevant context). This is the object of forthcoming investigations (\cite{Bar13}). 
\end{rem}

The paper is organized as follows. 

In Section~\ref{sec_preliminaries}, we fix notation; we give some details on the structure of the problem and we briefly recall the definitions and a few basic properties of the spaces and the operators we deal with, also discussing the capacity density condition and some properties of Sobolev functions in $p$-thick domains; we recall a parabolic covering lemma of Calder\'on-Zygmund type. 

In Section~\ref{sec_comparison}, we state and prove  
global higher regularity results for the solutions to the homogeneous analog of problem~\eqref{problema}, and other preliminary results. 
 
In Section~\ref{sec_proofs}, we prove the results stated in Theorem~\ref{thm_du}, 
and the Lorentz space estimates for the solutions~$u$.

\vspace{2mm}
\section{Preliminaries}\label{sec_preliminaries}

Throughout the paper we follow the usual convention of denoting by $c$ a general positive constant {\em always greater or equal than one} that may vary from line to line. Relevant dependencies will be emphasized by using parentheses; special constants will be denoted by $c_0$, $c_1$, $c_2$,....

As usual, we denote by 
$$
B_R(x_0)=B(x_0;R):=\{x\in \mathbb{R}^n : |x-x_0|<R\}
$$
the open ball centered in $x_0\in \mathbb{R}^n$ with radius $R>0$. When not important and clear from the context, we shall use the shorter notation $B_R=B_R(x_0)$. Let $(x_0,t_0)\in \mathbb{R}^{n+1}$, we denote by
\begin{equation}\label{bcyl}
\mathcal{Q}_R(x_0,t_0)=\mathcal{Q}(x_0,t_0;R):=B(x_0;R)\times (t_0-R^2,t_0)
\end{equation}
the open parabolic cylinder with height $R^2$ and having a ball $B_R$ as horizontal slice.

Finally, we mean $\partial_{{\rm par}}\Om_T=\partial_{{\rm lat}}\Om_T\cup \Omega_0$, where $\partial_{{\rm lat}}\Om_T:=\partial \Omega\times (-T,0)$ is the lateral boundary and $\Omega_0=\Omega\times\{-T\} $ is the initial boundary. 

\subsection{Solvability of the problem}\label{sec_solvability}
A measurable function $u$ is a very weak solution to~\eqref{problema} if $u\in L^1(-T,0; W^{1,1}_0(\Om))$ and
\begin{equation}\label{def_distrib}
-\int_{\Om_T} u\phi_t\,{\rm d}x\,{\rm d}t+\int_{\Om_T}\langle a(x,t,Du),D\phi\rangle\,{\rm d}x\,{\rm d}t=\int_{\Om_T}f\phi\,{\rm d}x\,{\rm d}t,
\end{equation}
for any $\phi\in C_0^\infty(\Om_T)$.
Also, while the lateral boundary condition can be formulated by prescribing the belonging of $u$ to $L^1(-T,0; W^{1,1}_0(\Om))$, the initial boundary condition is understood in the $L^1$-sense, that is
\begin{equation}\label{initial_value}
\frac 1 h \int_{-T}^{-T+h}\int_\Om|u(x,t)|\,{\rm d}x\,{\rm d}t\to 0\quad \text{as }h\to 0. 
\end{equation}
As usual in the parabolic setting, one can provide a convenient ``slice-wise" reformulation of equality \eqref{def_distrib} by mean of Steklov average. Indeed, for $h>0$ and $t\in [-T,0)$, we can define
$$
u_h(x,t):=
\begin{cases}
\displaystyle
\frac 1 h\int_{t}^{t+h}u(x,\tilde{t})\,{\rm d}\tilde{t} & \text{if} \ t+h\leq 0 \\[1.5ex]
0 & \text{if} \ t+h> 0
\end{cases}
$$
and the following equality
\begin{equation}\label{steklov}
\int_\Omega\big[(u_h)_t \phi+\langle[a(x,t,Du)]_h,D\phi\rangle\big]\,{\rm d}x=\int_\Om f_h\phi\,{\rm d}x 
\end{equation}
holds for any $\phi\in W_0^{1,2}(\Omega)$ and for a.~\!e.~$t\in(-T,0)$. This formulation allows to use the solution $u$, which enjoys only very weak regularity properties in time, as a test function in \eqref{steklov}, and therefore {\em throughout all the paper} we shall proceed formally in this sense, referring for instance to \cite{BH12} for the details in similar situations.
\vspace{1mm}

The existence of such solution is obtained using an approximation method; this means that we shall consider a special class of solutions called SOLA (Solutions Obtained by Limits of Approximations). In particular, one considers a sequence of bounded function $\{f_k\}\subset L^\infty(\Om_T)$ such that $f_k\to f$ in~$L^1(\Omega_T)$ as $k\to\infty$. Then, by standard monotonicity arguments, for each $k$ fixed, there exists a unique solution
$$
u_k\in C^0([-T,0]; L^2(\Omega))\cap L^2(-T,0; W^{1,2}_0(\Omega))
$$
and
\begin{equation*}
(u_k)_t\in C^{0}(-T,0;W^{-1,2}(\Om))
\end{equation*}
to the Cauchy-Dirichlet problem
\begin{equation}\label{problema_approx}
\begin{cases}
\displaystyle
(u_k)_t-\text{\rm div}\, a(x,t, Du_k) = f_k & \text{in} \ \Om_T \\
u_k = 0 & \text{on} \ \partial_{{\rm par}}\Om_T.
\end{cases}
\end{equation}
Finally, thanks to the arguments in \cite{BG89,BDGO97}, we can pass to the limit in the problem above and to prove the existence of a function $u\in L^1(-T,0; W^{1,1}_0(\Om))$ which solves \eqref{problema} in the distributional sense. For the rest of this paper we understand by $\{u_k\}$ the sequence obtained by solving \eqref{problema_approx} with
\begin{equation}\label{def_troncaf}
f_k(x,t):=\max\big\{-k,\ \min\{f(x,t),\, k\}\,\big\},\quad k\in\mathbb{N}. 
\end{equation}
Notice that, as stated in Remark \ref{remro}, this choice is not restrictive, since in this class of solutions uniqueness holds, in the sense that considering another approximating sequence $\{\bar f_k\}$ would lead to the same limit $u$.

\subsection{Relevant parabolic function spaces}\label{sec_spaces}

Let $\Om_T$ be the space time cylinder $\Om\times(-T,0)$, with $\Om$ being a bounded open set in $\R^n$, $n\geq 2$ and $T>0$.

Fix $q\in (0,\infty)$. A measurable map $g:\Om_T\to\R^k$ belongs to the {\it Lorentz space} $L(\gamma,q)(\Om_T)$ with $1\leq \gamma < \infty$ if and only if
\begin{equation}\label{def_lorentz}
\displaystyle
\|g\|^q_{L(\gamma, q)(\Om_T)} \, := \, q \int_0^\infty \left ( \lambda^\gamma \big| \big\{
(x,t)\in \Om_T : |g(x,t)|>\lambda \big\}\big| \right)^{\frac{q}{\gamma}} \frac{{\rm d}\lambda}{\lambda}
\, < \, +\infty.
\end{equation}
In the case $q=\infty$, the Lorentz space~$L(\gamma,\infty)$ with $\gamma\in [1,\infty)$ is the so-called {\it Marcinkiewicz space}
 and it is usually denoted by $\mathcal{M}^\gamma(\Om_T)$. A measurable map $g:\Om_T\to\R^k$ belongs to $\mathcal{M}^\gamma(\Om_T)$ if and only if
\begin{equation}\label{def_marci}
\|g\|^\gamma_{\mathcal{M}^\gamma(\Om_T)}  \equiv  \|g\|^\gamma_{L(\gamma,\infty)(\Om_T)} 
\, := \, \sup_{\lambda>0} \lambda^\gamma \big|\big\{ (x,t)\in \Om_T : |g(x,t)|>\lambda \big\}\big| \,
< \, +\infty.
\end{equation}

Note that by Fubini's Theorem one can see that, for any $\gamma\in [1,\infty)$, $L(\gamma,\gamma)(\Om_T)$ is nothing than the classic Lebesgue space $L^\gamma(\Om_T)$.
Also,  by Fatou's Lemma, one can see that the functionals~$\|\cdot\|_{L(\gamma,q)(\Om_T)}$ defined in~\eqref{def_lorentz}-\eqref{def_marci} are lower semi-continuous with respect to the a.~\!e.~convergence; \label{lcs}
see \cite[Section 3.2]{BH12b}.

\vspace{1mm}

Moreover, the spaces defined above enjoy H\"older type inequalities. In particular, below we state a standard inequality for the Marcinkiewicz spaces $\Mz^\gamma(\Om)$ in the form we will need it in the following of the paper.
\begin{lemma}
Let~$E\subseteq\R^{n+1}$ be a measurable set and let $g\in\Mz^\gamma(E)$ with~$\gamma>1$. Then, for any~$q\in[1,\gamma)$, $g\in L^q(E)$ and
$$
\displaystyle
\|g\|_{L^q(E)} \leq \left(\frac{\gamma}{\gamma-q}\right)^{\frac{1}{q}} |E|^{\frac{1}{q}-\frac{1}{\gamma}}\|g\|_{\Mz^\gamma(E)}.
$$
\end{lemma}

\vspace{1mm}

For details and results about the theory of Lorentz spaces, we refer the interested reader to~\cite{Ada75,Gra04}.

\subsection{Parabolic maximal operators}
For any measurable function~$g$, the restricted {\it fractional maximal operator}~$M^*_{\alpha,\qq_0}$, with~$\alpha\in [0,N]$, relative to a  parabolic cylinder $\qq_0$,  
is defined by
\begin{equation*}
\displaystyle
M^*_{\alpha, \qq_0}(g)(x,t) \, := \, 
\sup_{\qq\subseteq \qq_0, \ (x,t)\in \qq} |\qq|^{\frac{\alpha}{N}} \dashint_\qq |g(y,\tau)|\, {\rm d}y\, {\rm d}\tau,
\end{equation*}
where the cylinders $\qq$ have sides parallel to those of $\qq_0$.

The boundedness of maximal operators in Marcinkiewicz spaces is classical (see, for instance, \cite{Gra04}); i.~\!e., for any $g\in L^\gamma(\qq_0)$,
\begin{equation}\label{eq_8star}
\big|\big\{ (x,t) \in \qq_0 : M^*_{0,\qq_0}(g)(x,t)\geq \lambda \big\}\big|
\, \leq \, 
\frac{\bar{c}}{\lambda^\gamma}\int_{\qq_0}|g|^\gamma\,{\rm d}x\,{\rm d}t
\end{equation}
holds for every $\lambda>0$ and $\gamma\geq 1$; the constant $c$ depending only on $n$ and $\gamma$.

More in general, a standard embedding result for the maximal functions in Lorentz spaces does hold, as given by the following theorem, whose proof follows by the Marcinkiewicz theorem together with standard sublinear interpolation. 
\begin{thm}{\rm (\cite[Theorem 7]{Min10})}. 
Let $\alpha$ in $[0,N)$, and~$\gamma>1$ be such that~$\alpha\gamma<N$. Let ${\mathcal Q}\subset\R^{n+1}$ be a parabolic cylinder  and $q\in(0,\infty]$.
Then for every measurable function~$g$ in $L(\gamma,q)(\mathcal{Q})$ it holds
\begin{equation*}
\| M^*_{\alpha, {\mathcal Q}}(g)\|_{L\left(\frac{N\gamma}{N-\alpha\gamma}, q\right)({\mathcal Q})} \leq c  \|g\|_{L(\gamma,q)({\mathcal Q})},
\end{equation*} 
where $c$ is a constant depending only on $\alpha$, $\gamma$, $n$ and $q$.
\end{thm}
\vspace{1mm}

Also, in the borderline case $\gamma=1$, we have the following
\begin{thm}{\rm (\cite[Lemma 4.11 and Theorem 4.12]{BH12b})}.\label{thm_marci}
 Let $\alpha \in [0,N)$, ${\mathcal Q}\subset\R^{n+1}$ be a parabolic cylinder and consider the concentric parabolic cylinder $\sigma{\mathcal Q}$ scaled by a factor $\sigma>1$. Then there exists a constant $c =c(n,\alpha,\sigma)$ such that, for any measurable function $g$ in $\sigma{\mathcal Q}$, it holds
$$
\| M^*_{\alpha,{\mathcal Q}}(g)\|_{L^{\frac{N}{N-\alpha}}({\mathcal Q})} \, \leq \, c\, |{\mathcal Q}|^{1-\frac{\alpha}{N}} \|g\|^{\frac{\alpha}{N}}_{L^{1}(\sigma{\mathcal Q})} \|g\|^{1-\frac{\alpha}{N}}_{L\log L({\mathcal Q})}.
$$
\end{thm}

\subsection{Capacity density condition and thick domains}
\label{sec_thick}
In the following we recall the definition and some  properties of the $p$-thick domains.
We need first to recall the definition of $p$-capacity of sets.
\begin{defn}
Let $K\subseteq \Omega$ be a compact set. For any $p\geq 1$ the {\it variational $p$-capacity of $K$ on~$\Om$} is defined by
$$
{\rm cap}_p(K,\Om):=\inf\left\{\int_\Omega |D \varphi
|^p\,{\rm d}x : \varphi\in C^{\infty}_0(\Om), \varphi \geq \mathds{1}_K\right\}.
$$
\end{defn}
As customary, the definition above can be extended to generic open sets $A\subseteq\Om$ by taking the supremum over the capacities of the compact sets contained in $A$. 
In the case of  balls, the definition simplifies considerably and we have the following explicit formula for $B_\rho$ on $B_{2\rho}$
\begin{equation*}
{\rm cap}_p(\overline{{B}_\rho}, B_{2\rho})=c\,\rho^{n-p},
\end{equation*}
where $c$ is a positive constant depending only on $n$ and $p$.
For further details about the variational $p$-capacity, we refer for instance to Chapter 2 in~\cite{HKM93}.
\begin{defn}\label{def_p_thick}
A  set  $E\subset \mathbb{R}^n$  is {\it uniformly $p$-thick} (equivalently, satisfies the {\it $p$-capacity uniform thickness} condition) if there exist constants $c_0$ and $\rho_0>0$ such that
\begin{equation}\label{ere}
 {\rm cap}_p(E\cap \overline{{B}_\rho(x)}, B_{2\rho}(x))\geq c_0\,{\rm cap}_p(\overline{{B}_\rho(x)}, B_{2\rho}(x))=c(n,p,c_0)\,\rho^{n-p},
\end{equation}
for any $x\in E$ and for any $0<\rho<\rho_0$.
\end{defn}
For the sake of simplicity, from now on we shall use the short notation {\em thick} to denote a {\em $2$-thick} domain.

\vspace{1mm}

As stated in the introduction, one can prove that domains satisfying Definition~\ref{def_p_thick} include those with Lipschitz boundaries or even those that satisfy a uniform corkscrew condition. 
Furthermore, the estimate in~\eqref{ere} remains  valid for balls centered {\it outside} a uniformly $p$-thick domain (and near the boundary), as stated in the next lemma; see \cite[Lemma 3.8]{Par09}. This will be useful since we are going to deal with bounded sets {\em whose complementary} is 
 thick.
\begin{lemma}\label{boundary_p_thick}
Let $\Omega$ be a bounded domain such that $\mathbb{R}^n\setminus \Omega$ is uniformly $p$-thick with constants $c_0$ and $\rho_0$, and let $y\in\Om$ be such that $B_{\rho/6}(y)\setminus \Omega\neq \emptyset$. Then there exists a constant ${c_1}={c_1}(n,p,\rho_0,c_0)$ such that
\begin{equation*}
 {\rm cap}_p(\overline{{B}_{\rho/4}(y)}\setminus \Omega, B_{\rho/2}(y))\geq {c_1}\,{\rm cap}_p(\overline{{B}_{\rho/4}(y)}, B_{\rho/2}(y))=c\,\rho^{n-p}.
\end{equation*}
\end{lemma}
Also, a uniform $p$-thick domain enjoys a self-improving property, as stated in the following theorem due to Lewis \cite{Lew88} (see also Section 8 of \cite{Mik96} for a complete survey on the boundary regularity).

\begin{thm}\label{p_thick_dom_prop}
Let $\Omega$ be a bounded domain such that $\R^n\setminus\Omega$ is uniformly $p$-thick with constants $c_0$ and $\rho_0$, and $1<p\leq n$. Then there exists $q=q(n,p,c_0) \in (1,p)$ such that  $\R^n\setminus \Omega$ is uniformly $q$-thick with constant $c_2=c_2(n,p,c_0)$ and $\rho_0$. Moreover, $q$ can be chosen near $p$ so that $q\in(np/(n+p),p)$. Thus, we have
\begin{equation*}
{\rm cap}_q(\overline{{B}_{\rho/4}(y)}\setminus \Omega, B_{\rho/2}(y))\geq c_2\,{\rm cap}_q(\overline{{B}_{\rho/4}(y)}, B_{\rho/2}(y))\geq c(n,p,c_0)\rho^{n-q}
\end{equation*}
for $y\in\Om$ be such that $B_{\rho/6}(y)\setminus \Omega\neq \emptyset$.
\end{thm}

\vspace{1mm}
\begin{rem}
On the other hand a $p$-thick domain is $r$-thick for every $r\geq p$. Indeed, for $x\in E$, $\rho<\rho_0$ and $\varphi\in C^\infty_0(B_{2\rho}(x))$, $\varphi\geq \mathds 1_{E\cap \overline{{B}_\rho(x)}}$, by the H\"older inequality, we  have
\[
\int_{B_{2\rho}}|D\varphi|^p\dx\leq c(n,p,r)\,\biggl(\int_{B_{2\rho}}|D\varphi|^r\dx\biggl)^{\frac pr}\rho^{n\left(1-\frac pr\right)}.
\]
Taking the infimum for such $\varphi$ we get
\[
c(n,p,c_0)\,\rho^{n-p}\leq {\rm cap}_p(E\cap \overline{{B}_\rho}, B_{2\rho})\leq c\,\rho^{n\left(1-\frac pr\right)} \big[{\rm cap}_r(E\cap \overline{{B}_\rho}, B_{2\rho})\big]^{\frac pr}
\]
and $r$-thickness follows by taking the $r/p$-power. \end{rem}

\vspace{1.5mm}

We conclude this section by recalling some capacity Sobolev type inequalities; see for example \cite[Lemma 3.3]{KK94} or \cite[Lemma 8.11 and Remark 8.14]{Mik96}.
\begin{lemma}\label{Soboell}
Let $w\in W^{1,p}(B_{2\rho})$ and $K_\rho(w)=\{x\in \overline{{B}_\rho} : w(x)=0\}$. Define $\kappa=n/(n-p)$ if $p<n$ and $\kappa=2$ if $p\geq n$. Then there exists a positive constant $c=c(n,p)$ such that
$$
\left(\dashint_{B_{2\rho}}|w|^{\kappa p}\,{\rm d}x \right)^{\frac{1}{\kappa p}}\leq \left(\frac{c}{{\rm cap}_p(K_\rho(w), B_{2\rho})}\int_{B_{2\rho}}|Dw|^p\,{\rm d}x\right)^{\frac{1}{p}}.
$$
\end{lemma}
The previous inequality also holds in the case stated below, with $p$-powers on both sides; that is, we have the following Poincar\'e type inequality
\begin{lemma}\label{cap_Poinc}
Let $w\in W^{1,p}(B_{2\rho})$ and $K_\rho(w)$ as before. Then there exists a positive constant  $c\equiv c(n,p)$ such that
\begin{equation}\label{poci}
\dashint_{B_{2\rho}}|w|^p\,{\rm d}x \leq \frac{c}{{\rm cap}_p(K_\rho(w), B_{2\rho})}\int_{B_{2\rho}}|Dw|^p\,{\rm d}x. 
\end{equation}
\end{lemma}

\vspace{1.5mm}

Now a parabolic Sobolev inequality; its proof requires only slight changes with respect to the proof of Lemma~$3.17$ in~\cite{Par09}. Essentially, it is enough to apply Sobolev's inequality Lemma \ref{Soboell} slicewise to the continuation to zero of $w$ in $\mathcal Q_{\rho/2}\setminus\Omega_T$. Note moreover that the exponent $q$  appearing therein is the same present in Theorem \ref{p_thick_dom_prop}, for $p=2$, and that here it is unavoidable such a self-improving property. 
\begin{lemma}\label{par_Sobolev}
Let $\Omega\subset\R^n$ be a bounded domain such that $\R^n\setminus\Om$ is uniformly thick with constants $c_0$ and $\rho_0$, and let $\mathcal Q_{\rho/2}:=B_{\rho/2}(x_1)\times(t_1-(\rho/2)^2,t_1)$ be a parabolic cylinder such that $B_{\rho/6}(x_1)\setminus\Om\neq\emptyset$. 
Let moreover $w$ be a function such that
$$
w\in L^2(\mathcal Q_{\rho/2}\cap\Omega_T),\quad Dw\in L^2(\mathcal Q_{\rho/2}\cap\Omega_T),
$$
$$
w\in L^\infty(t_1-(\rho/2)^2,t_1; L^2(B_{\rho/2}(x_1)))
$$
 and $w\equiv0$ on $\mathcal Q_{\rho/2}\cap\partial_{\rm lat}\Omega_T$. Then there exist an exponent $q\equiv q(n,c_0)<2$ and a constant $c=c(n,c_0,\rho_0)$ such that
\begin{multline}
 \frac{1}{|\mathcal Q_{\rho/2}|}\int_{\mathcal Q_{\rho/2}\cap\Omega_T}|w|^2\dx\dt\leq c\,\biggl( \frac{1}{|\mathcal Q_{\rho/2}|}\int_{\mathcal Q_{\rho/2}\cap\Omega_T}|Dw|^q\dx\dt\biggr)^{\frac{2n}{q(n+2)}}\\
 \times\biggl(\sup_{\tau\in(t_1-(\rho/2)^2,t_1)\cap(-T,0)}\int_{B_{\rho/2}\cap\Om}|w(\cdot,\tau)|^2\dx\biggr)^{\frac{2}{n+2}}
\end{multline}
and
\begin{multline}
\int_{\mathcal Q_{\rho/2}\cap\Omega_T}|w|^{2(1+\frac 2n)}\dx\dt\leq c\, \int_{\mathcal Q_{\rho/2}\cap\Omega_T}|Dw|^2\dx\dt\\
 \times\biggl(\sup_{\tau\in(t_1-(\rho/2)^2,t_1)\cap(-T,0)}\int_{B_{\rho/2}\cap\Om}|w(\cdot,\tau)|^2\dx\biggr)^{\frac 2n}.
\end{multline}
\end{lemma}

\subsection{A parabolic covering lemma}
Below we present a technical lemma (see forthcoming Lemma~\ref{lem_covering}) which is nothing that a special version of the classical Calder\'on-Zygmund-Krylov-Safonov decomposition. It will allow us to work with parabolic cylinders that have as horizontal slice a ball instead of a cube. We would recall that a parabolic version of the Calder\'on-Zygmund-Krylov-Safonov covering lemma was proved by one of the authors in \cite{BH12b}, following the elliptic analog (which could be found, for instance, in~\cite[Lemma 1.2]{CP89}). Precisely, in ~\cite[Proposition 2.1]{BH12b} it is shown that a covering of a set by parabolic cylinders 
can be ``arranged'' in an optimal disjointed way so that in each of the relative horizontal slice cube the density is considerably small.
\vspace{1mm}

Since in the present paper we are interesting into analyzing some regularity properties of problem~\eqref{problema} up to the boundary, we need to work with parabolic cylinders having a ball as horizontal slice. For this, in order to prove Lemma~\ref{lem_covering} below, we will extend its elliptic counterpart given by Lemma~3.1 in~\cite{Phu12} to a parabolic framework.

\begin{lemma}\label{lem_covering}
Let $\mathcal{Q}_{\bar{R}}\subset \mathbb{R}^{n+1}$. Assume that $\mathcal{X}\subset \mathcal{Y}\subset \mathcal{Q}_{\bar{R}}$ are measurable sets such that there exists $0<\delta<1$ for which the following properties hold:
\begin{itemize}
\item[(i)]{
$|\mathcal{X}|<\delta |\mathcal{Q}_{\bar{R}}|$;
}\vspace{1mm}
\item[(ii)]{for all $(x,t)\in \mathcal{Q}_{\bar{R}}$ and for all $\rho\in (0, \bar{R}]$ if $|\mathcal{X}\cap \mathcal{Q}_\rho(x,t)|\geq \delta |\mathcal{Q}_\rho(x,t)|$ then $\mathcal{Q}_\rho(x,t)\cap \mathcal{Q}_{\bar{R}}\subset \mathcal{Y}$.
}\vspace{1mm}
\end{itemize}
Then we have 
$$
|\mathcal{X}|\leq c_3\,\delta|\mathcal{Y}|
$$
where $c_3=c_3(n)$.
\end{lemma}
\begin{proof}
Since (i) holds, for almost every $(x,t)\in \mathcal{X}$ there exists a $\rho_{x,t}< \bar{R}$ so that $|\mathcal{X}\cap \mathcal{Q}_{\rho_{x,t}}(x,t)|=\delta |\mathcal{Q}_{\rho_{x,t}}|$ and $|\mathcal{X}\cap \mathcal{Q}_{\rho}(x,t)|<\delta |\mathcal{Q}_{\rho}|$ for all $\rho_{x,t}<\rho<\bar{R}$.
Then, by the Vitali covering lemma we can extract a sequence of disjoint cylinders $\big\{Q_{\rho_{x_k,t_k}}(x_k, t_k)\big\}_k$ 
such that $\mathcal{X}\subset \left(\bigcup_k   \mathcal{Q}_{5\rho_{x_k,t_k}}(x_k,t_k) \right)\cap \mathcal{Q}_{\bar{R}}$. In view of the choice of $\mathcal{Q}_{\rho_{x_k,t_k}}(x_k,t_k)$, we infer that
$$
|\mathcal{X}\cap \mathcal{Q}_{5\rho_{x_k,t_k}}(x_k,t_k)|<\delta |\mathcal{Q}_{5\rho_{x_k,t_k}}|=5^{n+2}\delta |\mathcal{Q}_{\rho_{x_k,t_k}}|=5^{n+2}|\mathcal{X}\cap \mathcal{Q}_{\rho_{x_k,t_k}}(x_k,t_k)|.
$$ 
Moreover, there exists a constant $c(n)$ such that
$$
|\mathcal{Q}_{\rho_{x_k,t_k}}(x_k,t_k)|\leq c(n)|\mathcal{Q}_{\rho_{x_k,t_k}}(x_k,t_k) \cap \mathcal{Q}_{\bar{R}}|,
$$
since $(x_k, t_k)\in  \mathcal{Q}_{\bar{R}}$ and $\rho_{x_k, t_k}<\bar{R}$.

All in all, we have
\begin{eqnarray*}
|\mathcal{X}|=\left|\left(\cup_k   \mathcal{Q}_{5\rho_{x_k,t_k}}(x_k,t_k) \right)\cap \mathcal{X}\right|&\leq& \sum_k|\mathcal{Q}_{5\rho_{x_k,t_k}}(x_k,t_k)\cap \mathcal{X}|\\
&\leq& \sum_k5^{n+2}\delta|\mathcal{Q}_{\rho_{x_k,t_k}}(x_k,t_k)|\\
&\leq&c(n)5^{n+2}\sum_k \delta |\mathcal{Q}_{\rho_{x_k,t_k}}(x_k,t_k) \cap \mathcal{Q}_{\bar{R}}|\\
&\leq &c(n)5^{n+2}\delta|\mathcal{Y}|,
\end{eqnarray*}
where we used (ii) in the last inequality.
This concludes the proof. 
\end{proof}

\vspace{2mm}
\section{(Higher) regularity for the homogeneous problem and comparison results}\label{sec_comparison}

One of the key-points in the proofs of the main results in the present paper relies on obtaining both local interior and boundary comparison estimates. 
\vspace{1mm}

Let us first handle the interior ones by considering the  following homogeneous Cauchy-Dirichlet problem. 
\begin{equation}\label{problema_omog_int}
\begin{cases}
\displaystyle
w_t-\text{\rm div}\, a(x,t, Dw) = 0 & \text{in} \ \mathcal{Q}_{R} \\
w = u & \text{on} \ \partial_{{\rm par}}\mathcal{Q}_{R},
\end{cases}
\end{equation}
with $\mathcal{Q}_R\equiv \mathcal{Q}_R(x_0,t_0)\subset \Omega_T$ and $u$ being the unique solution to the regularized problem~\eqref{problema_approx}. We recall some basic results   from the higher integrability theory of Gehring. For the proofs we refer to \cite{Lie96}; see also \cite[Theorem~5.5]{BH12b}.

\begin{thm}\label{gehring}
Let $w\in C^0([t_0-R^2,t_0]; L^2(B_R(x_0)))\cap L^2(t_0-R^2,t_0; W^{1,2}(B_R(x_0)))$ be a weak solution to the parabolic equation \eqref{problema_omog_int}$_1$.
Then there exists $\chi_1=\chi_1(n,L, \nu)\break >1$ such that $Dw \in L^{2\chi_1}_{\rm loc}(\mathcal Q_R)$ and for any $q\in (0,2]$ it holds
\begin{equation}\label{stima_gehring}
\left(\dashint_{\mathcal Q_{R/2}} |Dw|^{2\chi_1}\,{\rm d}x\,{\rm d}t\right)^{\frac{1}{2\chi_1}}\, \leq\, c\left(\dashint_{\mathcal{Q}_R}|D w|^q\,{\rm d}x\,{\rm d}t\right)^{\frac 1q}.
\end{equation}
\\ Moreover, for any $q\in (0,2]$ and for any $\chi>1$ it holds
\begin{equation}\label{stima_gehringw}
\left(\dashint_{\mathcal Q_{R/2}} |w|^{2\chi}\,{\rm d}x\,{\rm d}t\right)^{\frac{1}{2\chi}}\, \leq\, c\left(\dashint_{\mathcal{Q}_R}|w|^q\,{\rm d}x\,{\rm d}t\right)^{\frac 1q}.
\end{equation}
The constants $c$ in~\eqref{stima_gehring} and \eqref{stima_gehringw} depend only on $n,L,\nu$ and $q$.
\end{thm}
Concerning local interior comparison estimates we can use those established in  \cite[Lemma 5.2]{BH12b} or \cite[Lemma 6.4]{BH12}. 
\begin{lemma}\label{comparison}
Let $u\in C^0([-T,0];L^2(\Omega))\cap L^2(-T,0; W^{1,2}_0(\Omega))$ be the unique solution to the regularized problem \eqref{problema_approx}. Moreover, let $w$ be the unique solution of the Cauchy-Dirichlet problem \eqref{problema_omog_int}. Then there exists a constant $c=c(n, \nu)$ such that
$$
\int_{\mathcal{Q}_R}\left(R^{-1}|u-w| + |Du-Dw|\right){\rm d}x\,{\rm d}t
\, \leq 
\, c\,R\int_{\mathcal{Q}_R} |f|\,{\rm d}x\,{\rm d}t. 
$$
\end{lemma}
For the local boundary estimates we consider $\mathcal{Q}_R=\mathcal{Q}_R(x_0,t_0)\subset \mathbb{R}^{n+1}$ such that it intersects the lateral boundary of $\Om_T$, that is 
\begin{equation*}
\mathcal{Q}_R \cap (\partial \Om\times(-T,0))
\, =\, \mathcal{Q}_R \cap \partial_{{\rm lat}}\Om_T\neq \emptyset.
\end{equation*}
This will be enough for our purposes, see  Remark \ref{rem1} at the beginning of Section~\ref{sec_proofs}. 
Then, we consider the unique solution to the following homogeneous Cauchy-Dirichlet problem 
\begin{equation}\label{problema_omog}
\begin{cases}
\displaystyle
w_t-\text{\rm div}\, a(x,t, Dw) = 0 & \text{in} \ \mathcal{Q}_R \cap\Omega_{T} \\
w = u & \text{on} \ \partial_{{\rm par}}(\mathcal{Q}_R \cap\Omega_{T}).
\end{cases}
\end{equation}
Here we need a higher integrability result for $w$, the counterpart of Theorem~\ref{gehring} up to the boundary, relying on some reverse H\"older inequalities. This kind of boundary higher integrability is already present in the literature (see for instance \cite[Theorem~4.7]{Par09}, and also \cite{Ark95,BDM10,Nik11} for more regular boundaries),
 but we prefer to propose an independent statement (and subsequently the relative proof) fitting our context.

\vspace{2mm}

For this, we need a couple of technical lemmata, the first of which is an appropriate version of the Gehring lemma, see \cite[Proposition 1.3]{GiaquintaStruwe:1982}:

\begin{lemma}\label{gehringLem}
Let $A\subset\R^{n+1}$ an open set and let $g:A\rightarrow \R^k$ be an integrable map such that
$$
	\dashint_{\qq_{\rho/8}}|g|^2\dx\dt
	\, \leq\, \hat c\,\biggl(\dashint_{\qq_\rho}|g|^q\dx\dt\biggr)^{\frac2q}+\vartheta\,\dashint_{\qq_\rho}|g|^2\dx\dt,
$$
for some constant $q<2$ and for all $\qq_\rho\subset A$. Then there exists a constant $\vartheta_0=\vartheta_0(n,q)$ such that if $\vartheta<\vartheta_0$ then $g\in L^{2\chi_1}_{\rm loc}(A)$ with $\chi_1>1$ depending on $n,\hat c,q,\vartheta$; moreover
$$
\biggl(	\dashint_{\qq_{\rho/2}}|g|^{2\chi_1}\dx\dt\biggr)^{\frac{1}{\chi_1}}\leq c\,\dashint_{\qq_\rho}|g|^2\dx\dt
$$
for all $\qq_\rho\subset A$. The constant $c$ also depends on $n,\hat c,q,\vartheta$.
\end{lemma}

The second one encodes the self-improving character of reverse H\"older inequalities. The proof follows \cite[Remark 6.12]{Giu03} and uses a covering argument.
\begin{lemma}\label{revHolder}
Let $A,g$ as in {\rm Lemma \ref{gehringLem}} and let
\[
{\biggl(\dashint_{\qq_{\rho/2}}{|g|}^{2\chi_1}\dx\dt\biggr)}^{\!\frac{1}{\chi_1}}
\, \leq\, \hat{c}\,\dashint_{\qq_\rho}|g|^2\dx\dt
\]
hold whenever $\qq_\rho\subset A$ is a cylinder, where $\chi_1>1$ and $\hat{c}\geq1$. Then, for every $\sigma\in(0,2]$, there exists a constant $c=c(n,\sigma, \hat{c})$ such that
\[
{\biggl(\dashint_{\qq_{\rho/2}}{|g|}^{2\chi_1}\dx\dt\biggr)}^{\!\frac{1}{2\chi_1}}
\,\leq\,  c\,{\biggl(\dashint_{\qq_\rho}|g|^\sigma\dx\dt\biggr)}^{\!\frac{1}{\sigma}}
\]
for every $\qq_\rho\subset A$.
\end{lemma}

Now, we are in a position to prove the following
\begin{thm}\label{gehring_boundary}
Let $\Om$ be a bounded domain such that $\mathbb{R}^n\setminus \Om$ satisfies a uniform thick condition with constants $c_0$ and $\rho_0$, and let $w$ be the unique solution to \eqref{problema_omog}. Then there exists a constants $\chi_1= \chi_1(n, \nu,L,\rho_0, c_0)>1$ such that for every $q>0$ there holds
\begin{equation}\label{stima_gehring_boundary}
\left(\frac{1}{|\mathcal{Q}_{R/2}|}\int_{\mathcal{Q}_{R/2} \cap\Omega_{T}}|Dw|^{2\chi_1}\,{\rm d}x\,{\rm d}t\right)^{\!\frac{1}{2\chi_1}} 
\, \leq \,
  c\,\biggl(\frac{1}{|\mathcal{Q}_R |}\int_{\mathcal Q_R\cap\Om_T}|Dw|^q\dx\dt\biggr)^{\frac1q}.
\end{equation} 
for a constant  $c= c(n, \nu,L,\rho_0, c_0,q)$. Moreover we also have
\begin{equation}\label{stima_gehring_boundary_u}
\left(\frac{1}{|\mathcal{Q}_{R/2}|}\int_{\mathcal{Q}_{R/2} \cap\Omega_{T}}|w|^{2\chi}\,{\rm d}x\,{\rm d}t\right)^{\!\frac{1}{2\chi}} 
\, \leq \,
  c\,\biggl(\frac{1}{|\mathcal{Q}_R |}\int_{\mathcal Q_R\cap\Om_T}|w|^q\dx\dt\biggr)^{\frac1q}
\end{equation} 
for any $\chi=\chi(n, \nu,L,\rho_0, c_0)>1$ and the same constant as above.
\end{thm}

\begin{proof}
The goal here is proving that if we consider the continuation of $w$ to zero (denoted by $\tilde  w$) in $\mathcal Q_R\setminus\Om_T$, then for every parabolic cylinder $\mathcal Q_\rho\equiv\mathcal Q_\rho(x_1,t_1)\subset \mathcal Q_R$ we have a reverse H\"older's inequality 
\begin{equation}\label{nasty}
 \biggl(\dashint_{\mathcal Q_{\rho/8}}|D\tilde w|^2\dx\dt\biggr)^{\frac{1}{2}}
 \leq  c\,\biggl(\dashint_{\mathcal Q_\rho}|D\tilde  w|^q\dx\dt\biggr)^{\frac{1}{q}}
 +\, \frac{\vartheta_0}{2} \biggl(\dashint_{\mathcal Q_{\rho}}|D\tilde  w|^2\dx\dt\biggr)^{\frac{1}{2}}, 
 \end{equation}
for some $q<2$ and for a constant depending on $n,\nu,L,\rho_0,c_0,\vartheta_0$, where $\vartheta_0$ is the constant of Lemma~\ref{gehringLem} corresponding to $q$. Therefore we could apply the Gehring lemma in the form of Lemma~\ref{gehringLem}, and then Lemma~\ref{revHolder}, to infer \eqref{stima_gehring_boundary}, after again restricting $\tilde  w$ to $\mathcal Q_R\cap\Om_T$. The point here is that we are not interested in the boundary higher integrability of $w$ in the part of the same boundary inside $\Omega_T$, but just in the higher integrability on cylinders centered in $\partial_{\rm lat}\Om_T$, where $w\equiv u\equiv0$.  In a very rough sense, this is similar to interior higher integrability. Hence now we fix a parabolic cylinder $\mathcal Q_\rho\subset \mathcal Q_R$; note that due to the assumption $\mathcal Q_R\cap\Om_0=\emptyset$, we also have  $\mathcal Q_\rho\cap\Om_0=\emptyset$.

\vspace{2mm}

The scheme of the proof is the following: first we consider  the case where $B_{\rho/6}\setminus\Om\neq\emptyset$. Here we first use a Caccioppoli type inequality to, roughly speaking, control the gradient $Dw$ over $B_{\rho/4}\cap\Om_T$ with the integral of $w$ over $B_{\rho/2}\cap\Om_T$. By applying the Sobolev inequality over $B_{\rho/2}\cap\Om_T$, we now need to estimate the supremum in \eqref{rig} in terms of $|Dw|$ and this is again achieved  with the help of a Caccioppoli type inequality and the Poincar\'e capacity inequality in Lemma \ref{cap_Poinc}. Finally, we get \eqref{nasty} by narrowing the set on the left-hand side. On the other hand, if $B_{\rho/6}\setminus\Om=\emptyset$, estimate \eqref{nasty} between $B_{\rho/8}$ and $B_{\rho/6}$ is just an interior one, and then we can enlarge the domain on the right-hand side.

\vspace{2mm}

{\bf Suppose now} hence {\bf $B_{\rho/6}\setminus\Om\neq\emptyset$} (but clearly also $B_{\rho/6}\cap\Om\neq\emptyset$, otherwise there would be nothing to prove) and consider a cut-off function $\varphi \in C^\infty_0(\mathcal Q_{R})$, $0\leq\varphi\leq1$ such that $\varphi\equiv1$ on $\qq_{\rho/4}$, ${\rm supp}\,\varphi\subset \qq_{\rho/2}$ and $|D\varphi|+\rho|\varphi_t|\leq c/\rho$. 
Note that from now on all the cylinders with radii multiple of $\rho$ will have vertex $(x_1,t_1)$. Up to a standard regularization of $w$ in time by Steklov averages, 
we use as a test function in \eqref{problema_omog} the map $\phi:=w\varphi^2\zeta_\varepsilon\mathds{1}_{\Om_T}$, 
where, once fixed $\tau \in (t_1-(\rho/2)^2,t_1)$, $\zeta_\varepsilon\equiv\zeta_{\varepsilon,\tau}(t)\in W^{1,\infty}(\R)$ is a piecewise linear continuous map, for $\varepsilon\leq (t_1-\tau)/2$, such that $\zeta_\varepsilon\equiv1$ in $(-\infty,\tau]$ and $\zeta_\varepsilon\equiv0$ in $[\tau+\varepsilon,+\infty)$. 
Note that for a.\!~e. $t\in (t_1-(\rho/2)^2,t_1)$, we have $\phi(\cdot,t)\in W^{1,2}_0(\Omega)$ and thus can be used in \eqref{steklov}. 
We therefore have, integrating over $(t_1-(\rho/2)^2,t_1)$,
\begin{multline}\label{cixi}
\int_{\mathcal Q_{\rho/2}\cap\Omega_T}w_tw\varphi^2\zeta_\varepsilon\dx\dt+\int_{\mathcal Q_{\rho/2}\cap\Omega_T}\langle a(x,t,Dw),Dw\rangle\varphi^2\zeta_\varepsilon\dx\dt\\=-2\int_{\mathcal Q_{\rho/2}\cap\Omega_T}\langle a(x,t,Dw),D\varphi\rangle w\varphi\zeta_\varepsilon\dx\dt.
\end{multline}
Now, we treat the first term in the following way, using integration by parts and recalling that $\varphi\zeta_\varepsilon\equiv0$ on $(B_{\rho/2}\cap\Om)\times\{t_1-(\rho/2)^2,t_1\}$,

\begin{align}\label{est.pr}
\int_{\qq_{\rho/2}\cap\Omega_T}w_tw&\varphi^2\zeta_\varepsilon\dx\dt=\frac12\int_{\qq_{\rho/2}\cap\Omega_T}\big(|w|^2\big)_t\varphi^2\zeta_\varepsilon\dx\dt\notag\\[1ex] 
&= -\frac12\int_{\qq_{\rho/2}\cap\Omega_T}|w|^2(\varphi^2\zeta_\varepsilon)_t\dx\dt\notag\\[1ex] 
&=-\int_{\qq_{\rho/2}\cap\Omega_T}|w|^2\varphi\varphi_t\zeta_\varepsilon\dx\dt+\frac{1}{2}\,\dashint_\tau^{\tau+\varepsilon}\!\!\int_{B_{\rho/2}\cap\Om}|w|^2\varphi^2\dx\dt.
\end{align}
Letting $\varepsilon\searrow0$ and then taking the (essential) supremum with respect to $\tau\in(t_1-(\rho/2)^2,t_1)$ we therefore get
\begin{multline}\label{elio}
\sup_{\tau\in(t_1-(\rho/2)^2,t_1)}\int_{\mathcal Q_{\rho/2,\tau}\cap\Omega_T}w_tw\varphi^2\dx\dt\geq-c\,\int_{\qq_{\rho/2}\cap\Omega_T}\frac{|w|^2}{\rho^2}\dx\dt\\+\frac12\sup_{\tau\in(t_1-(\rho/2)^2,t_1)}\int_{B_{\rho/2}\cap\Om}\big[|w|^2\varphi^2\big](\cdot,\tau)\dx.
\end{multline}
where we denoted for shortness with $\mathcal Q_{\rho/2,\tau}$ the set $\qq_{\rho/2}\cap\mathds{1}_{\{t<\tau\}}$.
The second term in the left-hand side of~\eqref{cixi} is estimated from below
\begin{equation}\label{est.dut}
\int_{\mathcal Q_{\rho/2}\cap\Omega_T}\langle a(x,t,Dw),Dw\rangle\varphi^2\zeta_\varepsilon\dx\dt
 \,\geq\,  \frac\nu2\int_{\mathcal Q_{\rho/2,\tau}\cap\Om_T}|Dw|^2\varphi^2\dx\dt 
\end{equation}
with $c= c(\nu,L)$, having at hand the following monotonicity condition, that can be deduced by the assumptions in~\eqref{hp_a}:
$$
\langle a(x,t,z),z\rangle\, \geq\,\frac\nu2|z|^2 \ \, \text{for any} \ (x,t)\in\Omega_T \ \text{and any} \ z\in\R^n.
$$

For the latter term in the right-hand side of \eqref{cixi} we have, using again the growth condition \eqref{hp_a}$_2$ together with the Young inequality,
\begin{eqnarray}\label{est.tre}
&&\!\!\!\!\!\!\!\!\!\!\!\!\!\!\!\! 2\int_{\mathcal Q_{\rho/2}\cap\Omega_T}\langle a(x,t,Dw),D\varphi\rangle w\varphi\zeta_\varepsilon\dx\dt \nonumber\\
 &&\qquad \leq \frac\nu4\int_{\Omega_T\cap \mathcal Q_{\rho/2,\tau+\varepsilon}}|Dw|^2\varphi^2\dx\dt +\frac{c(\nu,L)}{\rho^2}\int_{\mathcal Q_{\rho/2}\cap\Omega_T}|w|^2\dx\dt.
\end{eqnarray}
To conclude we put \eqref{est.pr}, \eqref{est.dut} and \eqref{est.tre} into \eqref{cixi}, we let $\varepsilon\searrow0$, we perform some algebraic manipulations and we take the supremum with respect to $\tau\in (t_1-(\rho/2)^2,t_1)$. Using also \eqref{elio} finally we get
\begin{eqnarray}\label{rh1}
&& \sup_{\tau\in(t_1-(\rho/4)^2,t_1)}\int_{ B_{\rho/4}\cap\Om}|w(\cdot,\tau)|^2\dx+ \int_{\mathcal Q_{\rho/4}\cap\Omega_T}|Dw|^2\dx\dt \nonumber\\[0.5ex]
&& \qquad  \qquad  \qquad  \qquad   \qquad  \qquad \qquad  \leq\frac{c(\nu,L)}{\rho^2}\,\int_{\mathcal Q_{\rho/2}\cap\Omega_T}|w|^2\dx\dt,
\end{eqnarray}
recalling that $\varphi\equiv1$ on $\mathcal Q_{\rho/4}$. Note that also a completely similar estimate, where the cylinder $\mathcal Q_{\rho/2}$ appears on the left-hand side, and the cylinder $\mathcal Q_\rho$ on the right-hand side, follows by straightforward changes in the proof above.

\vspace{3mm}

We recall now we are supposing $B_{\rho/6}\setminus\Om\neq\emptyset$; therefore we can use the Sobolev inequality in Lemma \ref{par_Sobolev} in the right-hand side of~\eqref{rh1} and, after dividing by $|\mathcal Q_{\rho/2}|$, we infer
\begin{eqnarray}\label{rig}
&& \frac{c}{|\mathcal Q_{\rho/4}|}\int_{\mathcal Q_{\rho/4}\cap \Omega_T}|Dw|^2\dx\dt \nonumber \\
&& \qquad\qquad\qquad  \leq \, c\,\biggl( \frac{1}{|\mathcal Q_{{\rho}/{2}}|}\int_{\mathcal Q_{{\rho}/{2}}\cap\Omega_T}|Dw|^q\dx\dt\biggr)^{\frac{2n}{q(n+2)}}  \\[0.5ex]
&& \qquad\qquad\qquad \quad  \, \times
\biggl(\sup_{\tau\in(t_1-({\rho}/{2})^2,t_1)}\frac{1}{|\mathcal Q_{{\rho}/{2}}|}\int_{B_{\rho/2}\cap\Om}|w(\cdot,\tau)|^2\dx\biggr)^{\frac{2}{n+2}} \nonumber
\end{eqnarray}
where $c$ depends upon $n,\nu,L,\rho_0,c_0$. Finally, we estimate the supremum on the right-hand side in the following way: for a.~\!e. $\tau \in (t_1-(\rho/2)^2,t_1)$ using \eqref{rh1} in its version over $\mathcal Q_{\rho/2}$ and $\mathcal Q_{\rho}$ and Poincar\'e's inequality~\eqref{poci} slice-wise, for $p=2$, we have
\begin{eqnarray*} 
 \int_{B_{\rho/2}\cap\Om}|w(\cdot,\tau)|^2\,\dx
 \! & \leq & \! \frac{c}{\rho^2}\,\int_{\mathcal Q_{\rho}\cap \Omega_T}|w|^2\dx\dt \nonumber
 \\[1ex]
& \leq & \! \frac{c}{\rho^2}\,\int_{t_1-\rho^2}^{t_1}\frac{\rho^n}{{\rm cap}_2(K_{\rho/2}(w),B_{\rho})}\int_{B_{\rho}}|Dw|^2\dx\dt \nonumber 
\end{eqnarray*}
where $K_{\rho/2}(w)=\{x\in \overline{B_{\rho/2}}:w(x)=0\}$. We recall that  we extended $w$ to zero in $\mathcal Q_R\setminus\Om_T$; note this gives $w\in W^{1,2}(B_R)$ for a.~\!e. $t$ since $w\equiv 0$ on $\partial_{\rm lat}\Omega_T$. Now, since $K_{\rho/2}(w)\supset \overline{B_{\rho/2}}\setminus\Omega$ and $B_{\rho/6}\setminus\Om\neq\emptyset$ we can apply Lemma~\ref{boundary_p_thick} to get
\[
{\rm cap}_2(K_{\rho/2}(w), B_{\rho})
\, \geq\, {c_1}\,{\rm cap}_2(\overline{B_{\rho/2}}, B_{\rho})
\,= \,c\,\rho^{n-2}, 
\]
and therefore, again dividing by $|\mathcal Q_{\rho/2}|$,
\[
\sup_{\tau\in(t_1-(\rho/2)^2,t_1)} \frac{1}{|\mathcal Q_{\rho/2}|}\int_{B_{\rho/2}\cap\Om}|w(\cdot,\tau)|^2\dx \leq \frac{c}{|\mathcal Q_{\rho/2}|}\int_{\mathcal Q_\rho\cap \Omega_T}|Dw|^2\dx\dt.
\]
Putting the latter estimate into \eqref{rig} and using the Young inequality with conjugate exponents $(n+2)/2$ and $(n+2)/n$, with $\vartheta_0$ being the constant appearing in Lemma~\ref{gehringLem} corresponding to the choice of $q$ following from Lemma~\ref{par_Sobolev},  we infer
\begin{eqnarray*}
 &&\!\!\!\!\!\frac{c}{|\mathcal Q_{\rho/4}|}\int_{\mathcal Q_{\rho/4}\cap \Omega_T}|Dw|^2\dx\dt \nonumber \\[1ex]
 && \ \  \leq  c\,\biggl( \frac{1}{|\mathcal Q_{\rho/2}|}\int_{\mathcal Q_{\rho/2}\cap\Omega_T}|Dw|^q\dx\dt\biggr)^{\frac{2n}{q(n+2)}}
 \biggl( \frac{1}{|\mathcal Q_{\rho/2}|}\int_{\mathcal Q_\rho\cap \Omega_T}|Dw|^2\dx\dt\biggr)^{\frac{2}{n+2}}
 \notag\\[1ex]
 && \ \  \leq  c(\vartheta_0)\,\biggl( \frac{1}{|\mathcal Q_{\rho/2}|}\int_{\mathcal Q_{\rho/2}\cap\Omega_T}|Dw|^q\dx\dt\biggr)^{\frac{2}{q}}\\
 & &\ \ \quad +\displaystyle \Bigl(\frac{\vartheta_0}{100}\Bigr)^2 \frac{1}{|\mathcal Q_\rho|}\int_{\mathcal Q_{\rho}\cap \Omega_T}|Dw|^2\dx\dt. \notag
\end{eqnarray*}
The inequality above plainly yields~\eqref{nasty} in the case $B_{\rho/6}\setminus\Om\neq\emptyset$.

\vspace{3mm}

 {\bf The case $B_{\rho/6}\setminus\Om=\emptyset$} is easier, since now an estimate 
\[
 \biggl(\dashint_{\mathcal Q_{\rho/8}}|D\tilde  w|^2\dx\dt\biggr)^{\frac{1}{2}}\leq c\,\biggl(\dashint_{\mathcal Q_{\rho/6}}|D\tilde  w|^q\dx\dt\biggr)^{\frac{1}{q}},
\]
$q$ possibly different from the one appearing in \eqref{rig}, is an interior one and can be deduced following \cite{GiaquintaStruwe:1982} or \cite{KL}. At this point we just enlarge the integral on the right-hand side to get \eqref{nasty}. Note that actually higher integrability estimate~\eqref{stima_gehring_boundary} holds for every couple of cylinders $\qq_{\rho/4}$, $\qq_{\rho/2}$ such that $\qq_\rho\subset\qq_R$, and we shall use this fact in few lines.

\vspace{3mm}
 
Finally, the estimate \eqref{stima_gehring_boundary_u} for~$w$ is a straightforward consequence of 
boundary sup-estimate for parabolic problems with null boundary datum. For this, we refer for instance to \cite[Theorem 6.30]{Lie96} in a slightly different, linear setting. For the sake of the reader we also propose in the following lemma a proof adapted to our framework. 
\end{proof}

\begin{lemma}
 Let $w$ be the solution of the Cauchy-Dirichlet problem~\eqref{problema_omog}. Then for every $q>0$ it holds
\begin{equation}\label{sup.u}
 \sup_{\qq_{R/2}\cap \Om_T}|w|\leq c\,\biggl(\frac{1}{|\qq_R|}\,\int_{\qq_R\cap\Om_T}|w|^q\dx\dt\biggr)^{\frac1q},
\end{equation}
 for a constant $c$ depending $n,\nu,L,c_0,\rho_0,q$.
\end{lemma}
\begin{proof}
Consider, for $R/2\leq\rho\leq R$ and $h\in\R$, the quantities
\begin{eqnarray}\label{deffss}
 U(h,\rho):=\int_{\qq_\rho\cap\Omega_T}\big|(w-h)_+\big|^2\dx\dt, \nonumber \\[-0.7ex]
 \\[-0.7ex]
 V(h,\rho):=|\qq_\rho\cap\Omega_T\cap\{w\geq h\}|=:|A(h,\rho)|; \nonumber
\end{eqnarray}
here $s_+$ is the positive  parts of $s$; here all the cylinders share the same center~$(x_0,t_0)$. For $\rho_1<\rho_2$ both belonging to $[R/2,R]$ and with $h<k$, we have the relations
\begin{align}\label{schi}
V(k,\rho_1)&\leq\frac{1}{(k-h)^2}U(h,\rho_2),\notag\\ 
U(k,\rho_1) &\leq c\,\frac{1}{(\rho_2-\rho_1)^2}U(h,\rho_2)\big[V(h,\rho_2)\big]^{\frac{2}{n+2}}.
\end{align}
The first inequality plainly follows by the involved definitions. Indeed,
\begin{equation}\label{eq_stare}
V(k,\rho_1)(k-h)^2\leq\int_{A(k,\rho_1)}(w-h)_+^2\dx\dt\leq\int_{A(h,\rho_1)}(w-h)_+^2\dx\dt.
\end{equation}
For the second one, we need some work. First of all, minor modifications in the proof of Caccioppoli's inequality \eqref{rh1} give for all $i\in\mathds{N}$
\begin{eqnarray}\label{Cacde}
&& \!\!\!\!\!\!\!\!\!\!\!\!\!\!\!\!\!\!\!\!\!\!\!\displaystyle \sup_{\tau\in(t_0-(\rho_1+\rho_2)^2/4,t_0)} \int_{B_{(\rho_1+\rho_2)/2}\cap\Om}\big|(w-k)_+(\cdot,\tau)\big|^2\dx\notag\\
&&\qquad \qquad\qquad\qquad+ \int_{\qq_{(\rho_1+\rho_2)/2}\cap\Om_T}\big|D(w-k)_+\big|^2\dx\dt \notag\\[1ex]
&&\qquad\qquad\qquad\qquad \quad \leq \displaystyle  \frac{c(\nu,L)}{(\rho_2-\rho_1)^2}\,\int_{\qq_{\rho_2}\cap\Om_T}\big|(w-k)_+\big|^2\dx\dt
\end{eqnarray}
for every $h\in\R$; in particular we need to choose here the test function 
\[
\phi:=(w-k)_+\varphi^2\zeta_\varepsilon\mathds{1}_{\Om_T},
\]
with $\varphi \in C^\infty_0(\mathcal Q_{R})$, $0\leq\varphi\leq1$ such that $\varphi\equiv1$ on~$\qq_{(\rho_1+\rho_2)/2}$, ${\rm supp}\,\varphi\subset \qq_{\rho_2}$ and $|D\varphi|+(\rho_2-\rho_1)|\varphi_t|\leq c/(\rho_2-\rho_1)$. 

\vspace{3mm}
Now, consider a cut-off function $\eta\in C^\infty_0(\qq_{(\rho_1+\rho_2)/2})$ such that $0\leq\eta\leq1$, $\eta\equiv1$ on $\qq_{\rho_1}$ and $|D\eta|\leq c/(\rho_2-\rho_1)$. By H\"older's inequality and the fact that $|\eta|\leq 1$
\begin{multline*} 
\int_{\qq_{\rho_1}\cap\Om_T}\big|(w-k)_+\big|^2\dx\dt \leq c(n)\,\int_{\qq_{(\rho_1+\rho_2)/2}\cap\Om_T}\big|(w-k)_+\eta\big|^2\dx\dt\\
 \leq c(n)\,\biggl(\int_{\qq_{(\rho_1+\rho_2)/2}\cap\Om_T}\big|(w-k)_+\eta\big|^{2(1+\frac2n)}\dx\dt\biggr)^{\frac{n}{n+2}}\big[V(\rho_2,k)\big]^{\frac{2}{n+2}}.
\end{multline*}
To estimate the right-hand side we use Sobolev's embedding Lemma \ref{par_Sobolev} and then twice Caccioppoli's inequality \eqref{Cacde},
\begin{multline*} 
\int_{\qq_{(\rho_1+\rho_2)/2}\cap\Om_T}\big|(w-k)_+\eta\big|^{2(1+\frac2n)}\dx\dt\\
\leq c\, \biggl(\sup_{\tau\in(t_0-(\rho_1+\rho_2)^2/4,t_0)}\int_{B_{(\rho_1+\rho_2)/2}\cap\Om}\big|(w-k)_+(\cdot,\tau)\big|^2\dx\biggr)^{\frac2n}\\
 \times \biggl(\int_{\qq_{\rho_2}\cap\Om_T}\big|D(w-k)_+\big|^2\dx\dt+\frac{1}{(\rho_2-\rho_1)^2}\int_{\qq_{\rho_2}\cap\Om_T}\big|(w-k)_+\big|^2\dx\dt\biggr)\\
 \leq c\,\biggl(\frac{1}{(\rho_2-\rho_1)^2}\int_{\qq_{\rho_2}\cap\Om_T}\big|(w-k)_+\big|^2\dx\dt
 \biggr)^{1+\frac2n};
\end{multline*}
we also used here the estimates for $|\eta|$ and $|D\eta|$. We estimate the first part of the latter term from above, keeping in mind that $h<k$ and therefore $A(k,\rho_2)\subset A(h,\rho_2)$
\begin{eqnarray*}
\int_{\qq_{\rho_2}\cap\Om_T}\big|(w-k)_+\big|^2\dx\dt
 & \leq & \int_{A(\rho_2,k)}\big|(w-h)_+\big|^2\dx\dt \\
& \leq & 
\int_{\qq_{\rho_2}\cap\Om_T}\big|(w-h)_+\big|^2\dx\dt.
\end{eqnarray*}
Finally, it suffices to merge all the estimates above together with~\eqref{eq_stare} (by replacing $\rho_1$ with $\rho_2$ there) to obtain the inequality in~\eqref{schi}.
\vspace{0,5mm}

At this point, we define the sequence of level sets and radii
\[
0< k_i:=2\Bigl(1-\frac{1}{2^i}\Bigr)k_1,\qquad\rho_i:=\frac{R}{2}+\frac{R}{2^i}
\]
where $k_1$ is a positive constant which will be defined in few lines. We moreover choose into \eqref{schi}, for $i\in\mathds{N}$ fixed, $\rho_1\equiv\rho_{i+1}$, $\rho_2\equiv\rho_i$, $k_i\equiv h <k\equiv k_{i+1}$. Call $U_i:=U(k_i,\rho_i)$ and $V_i:=V(k_i,\rho_i)$. 
By taking the $\alpha$-th power, $\alpha\geq0$, of \eqref{schi}$_2$ and the $\beta$-th power, $\beta\geq0$, of \eqref{schi}$_1$, and multiplying the resulting inequalities we finally have
$$
U_{i+1}^\alpha V_{i+1}^\beta 
\, \leq \,
c(\alpha,\beta) \frac{\,2^{2(\alpha+\beta)i}}{k_1^{2\beta} R^{2\alpha}}
U_i^{\alpha+\beta}V_i^{\frac{2\alpha}{n+2}},
$$
where $c(\alpha,\beta)$ depends also on $n,\nu,L,c_0,\rho_0$. Looking for $\vartheta>1$ such that $\alpha+\beta=\vartheta\alpha$, ${2\alpha}/{(n+2)}=\vartheta\beta$ gives the value 
\[
\vartheta=\frac12+\sqrt{\frac14+\frac{2}{n+2}};
\]
therefore, fixing $\alpha=(n+2)\vartheta/2$ (and therefore $\beta=1$), we have
\[
U_{i+1}^\alpha V_{i+1}
\, \leq \,
\tilde c\,\frac{2^{2\vartheta\alpha i}}{|\qq_R|^\vartheta k_1^2}
\Bigl(U_i^\alpha V_i\Bigr)^{\vartheta}. 
\]
A well known iteration argument, see  \cite[Lemma 7.1]{Giu03}, ensures that 
\begin{equation}\label{eqas}
\lim_{i\to \infty} U_i^\alpha V_i=0\qquad\text{if}\qquad \big[U_1^\alpha V_1\big]^{\vartheta-1}\leq \frac{2^{-\frac{\vartheta\alpha}{\vartheta-1}}}{\tilde c}  
|\qq_R|^{\vartheta}k_1^2.
 \end{equation}
Since, recalling that $\rho_i\geq R/2$ and $k_i\leq k_1$,
\[
0\leq \biggl(\int_{\qq_{R/2}\cap\Om_T}\big|(w-k_1)_+\big|^2\dx\dt\biggr)^\alpha|\qq_{R/2}\cap\Om_T\cap\{w\geq k_1\}|\leq U_i^{\alpha}V_i\longrightarrow0;
\]
this yields $w_+\leq k_1$ almost everywhere in $\qq_{R/2}\cap\Om_T$. We now find an appropriate choice of $k_1$ ensuring that the condition in \eqref{eqas} is satisfied. Since 
\[
\vartheta-1=\sqrt{\frac14+\frac{2}{n+2}}-\frac12
=\frac{2}{n+2}\frac{1}{\vartheta}=\frac1\alpha,
\]
recalling the definitions in \eqref{deffss} and the fact that $\rho_1=R$, we have
\begin{align*}
\big[U_1^\alpha V_1\big]^{\vartheta-1}&=\biggl(\int_{\qq_R\cap\Om_T}\big|(w-k_1)_+\big|^2\dx\dt\biggr)^{\alpha(\vartheta-1)}|\qq_R\cap\Om_T\cap\{w\geq k_1\}|^{\vartheta-1}\\
& \leq |\qq_R\cap\Om_T|^{\vartheta-1}\int_{\qq_R\cap\Om_T}(w_+)^2\dx\dt \\
& \leq  |\qq_R|^{\vartheta-1}\int_{\qq_R\cap\Om_T}(w_+)^2\dx\dt\\
&\leq \frac{2^{-\frac{\vartheta\alpha}{\vartheta-1}}}{\tilde c} |\qq_R|^\vartheta k_1^2,
\end{align*}
if we choose as starting level
\[
k_1^2:=\frac{2^{\frac{\vartheta\alpha}{\vartheta-1}}\,\tilde c}{|\qq_R|}\int_{\qq_R\cap\Om_T}|w|^2\dx\dt=\frac{c}{|\qq_R|}\int_{\qq_R\cap\Om_T}|w|^2\dx\dt,
\]
$c\equiv c(n,\nu,L,c_0,\rho_0)$. A similar argument, considering the negative part instead of the positive one, yields $w_-\leq k_1$. Therefore we showed the validity of \eqref{sup.u}, but with the $L^2$ norm on the right-hand side. We show here how to lower this exponent, following \cite[Chapter 7]{Giu03}. First notice that the estimate 
\[
 \sup_{\qq_{r_1}\cap \Om_T}|w|\leq c\,\biggl(\frac{1}{(r_2-r_1)^{n+2}}\int_{\qq_{r_2}\cap\Om_T}|w|^2\dx\dt
 \biggr)^{\frac12},
\]
for all $R\leq r_1<r_2\leq R$ holds as well; this can be seen with slight modifications of the previous proof, or adapting \cite[Corollary 7.1]{Giu03} to our setting. Note that this last approach in our case would require a more involved argument: we should indeed distinguish the cases as in the first part of the proof of Theorem~\ref{gehring_boundary}, and change a bit the radii. At this point, with this estimate at hand, we argue as follows: denote by $\mathcal{U}(r)$ the supremum of $|w|$ over $\qq_r\cap\Om_T$, for $R/2\leq r\leq R$. Then for $R/2\leq r_1<r_2\leq R$ and $q>0$, using Young's inequality
\begin{align*}
\mathcal{U}(r_1)&\leq c\,\biggl(\frac{1}{(r_2-r_1)^{n+2}}\int_{\qq_{r_2}\cap\Om_T}{|w|}^q\dx\dt\biggr)^{\frac12}\big[\mathcal{U}(r_2)\big]^{\frac{2-q}{2}}\\
&\leq \frac12\, \mathcal{U}(r_2)+c\,\biggl(\frac{1}{(r_2-r_1)^{n+2}}\int_{\qq_{R}\cap\Om_T}|w|^q\dx\dt\biggr)^{\frac1q}.
\end{align*}
A well known iteration lemma (see, e.~\!g., \cite[Lemma 6.1]{Giu03}) at this point finally gives~\eqref{sup.u}. 
\end{proof}

We conclude this section by stating  the counterpart of Lemma~\ref{comparison} up to the boundary. The proof plainly follows by that of Lemma~4.1 in \cite{KM12c} (see also~\cite[Lemma 3.3]{DP12b}), once taking into account the modifications needed in order to handle the fact that the cylinder intersects the boundary, see \cite[Lemma 2.6]{Phu12} and the proof of  previous Theorem~\ref{gehring_boundary}.

\begin{lemma}\label{comparison_boundary}
Let $u\in C^0([-T,0];L^2(\Omega))\cap L^2(-T,0; W^{1,2}_0(\Omega))$ be the unique solution to the regularized problem \eqref{problema_approx}, and let $w$ 
be the unique solution of the Cauchy-Dirichlet problem~\eqref{problema_omog}. Then there exists a positive constant depending only on $n,\nu$ such that
$$
\int_{\mathcal{Q}_R\cap \Om_T} \left( R^{-1}|u-w| + |Du-Dw|\right)\,{\rm d}x\,{\rm d}t
\,\leq \,c\,R\int_{\mathcal{Q}_R\cap \Om_T} |f|\,{\rm d}x\,{\rm d}t.
$$
\end{lemma}

\vspace{2mm}

\section{Proof of the main results}\label{sec_proofs}

This section is devoted to the proof of the main result, to the regularity analysis of the borderline case, and to precise Lorentz space estimates up to the boundary for the solution $u$. First, however, a couple of remarks allowing to simplify the proofs.

\begin{rem}\label{rem1}
Note that we can extend a solution to problem \eqref{problema} to $\Omega_{2T}:=\Om\times(-2T,0)$, where we also extend the function $f$ to zero in $\Om\times(-2T,-T]$ and the vector field $a$ in the following way: $\tilde{a}(x,t,z):=a(x,-2T-t,z)$ for $t\in(-2T,-T)$. Therefore, denoting respectively by $\tilde u,\tilde f$ the extension, we have that $\tilde u$ satisfies
\[
\begin{cases}
\displaystyle
\tilde u_t-\text{\rm div}\, \tilde a(x,t, D\tilde u) = \tilde f(x,t) \quad& \text{in} \ \Omega_{2T},\\
\tilde u = 0 & \text{on} \ \partial_{{\rm par}}\Om_{2T}.
\end{cases}
\]
Indeed, we briefly sketch  this argument. Given a test function $\varphi\in C^\infty_0(\Om_{2T})$, and supposing $\,{\rm supp}\,\varphi\cap\Om\times\{-T\}\neq\emptyset$ \,in order to avoid trivialities, it is clearly enough to prove that the distributional formulation \eqref{def_distrib} is satisfied in $C^{\infty}_{0}(\Omega_T)$  by the restriction of $\varphi$ in $\Omega_T$, i.~\!e.
\begin{equation}\label{part}
-\int_{\Om_T} u\varphi_t\,{\rm d}x\,{\rm d}t+\int_{\Om_T}\langle a(x,t,Du),D\varphi\rangle\,{\rm d}x\,{\rm d}t=\int_{\Om_T}f\varphi\,{\rm d}x\,{\rm d}t, 
\end{equation}
with $\varphi$ as above. To this aim, take a $C^\infty(\R)$ function $\zeta_\varepsilon$ such that $\zeta_\varepsilon\equiv1$ on $[-T+2\varepsilon,+\infty)$, $\zeta_\varepsilon\equiv 0$ on $(-\infty,-T+\varepsilon]$ and $|\zeta_\varepsilon'|\leq c/\varepsilon$, and put as a test function into \eqref{def_distrib} the product $\phi:=\varphi\zeta_\varepsilon\in C^\infty_0(\Omega_T)$. We get
\begin{multline*}
-\int_{\Om_T} u\varphi_t\zeta_\varepsilon\,{\rm d}x\,{\rm d}t-\int_{\Om_T} u\varphi\zeta_\varepsilon'\,{\rm d}x\,{\rm d}t+\int_{\Om_T}\langle a(x,t,Du),D\varphi\rangle\zeta_\varepsilon\,{\rm d}x\,{\rm d}t\\=\int_{\Om_T}f\varphi\zeta_\varepsilon\,{\rm d}x\,{\rm d}t.
\end{multline*}
Now, we let $\varepsilon$ go to zero. Notice that both the first and the third integral on the left-hand side converge to the corresponding integrals in~\eqref{part}; the same happens for the integral on the right-hand side. It remains to consider the contribution involving the derivative of $\zeta_\varepsilon$; for this, we have
\[
\biggl|\int_{\Om_T} u\varphi\zeta_\varepsilon'\,{\rm d}x\,{\rm d}t\biggr|
\, \leq \, c\,\|\varphi\|_{L^\infty(\Omega_T)}\dashint_{-T+\varepsilon}^{-T+2\varepsilon}\int_\Om|u(x,t)|\dx\dt\xrightarrow[
\varepsilon\to0]{}0
\]
thanks to~\eqref{initial_value}. This completes the claim.

\vspace{3mm}

This will allow to treat, by a simple argument, only cylinders not intersecting the initial boundary. Indeed for these, it would be enough to extend the solution as described above, and treat them as they were interior cylinders, or eventually cylinders just intersecting the lateral boundary.
\end{rem}

\begin{rem}
We stress here that we shall also not take into account the terminal boundary. Indeed the arguments in~\cite{BH12b} could be also performed using not symmetric, but backward parabolic cylinders as those in \eqref{bcyl}. This would apply in particular to parabolic cylinders of the form $\mathcal Q_\rho(x_0,0)$, with $B_\rho(x_0)\subset\Om$.
\end{rem}

\subsection{Integrability of $Du$}

We firstly need  a preliminary lemma, whose proof can be deduced by extending the original proof in the local case in~\cite[Lemma 6.2]{BH12b} (see, also,~\cite[Propositions 3.2 and 3.4]{Phu12}), which will permit us to apply the Calder\'on-Zygmund type Lemma~\ref{lem_covering} to the level sets of maximal operator functionals applied to the gradient of the solutions to our problem~\eqref{problema}. Since the novelty here is given by the up to the boundary estimates, we prefer to state this result in the following form, that is in accordance with the analog in the elliptic case given by Proposition 3.4 in~\cite{Phu12}.

\begin{lemma}\label{lem_37}
Let  
$u$ be the solution to~\eqref{problema_approx}. Let $A>1$ be an absolute constant depending only on $n$, $L$, $\nu$, $c_0$, $\rho_0$, and let $\chi_1>1$ be the higher integrability exponent as in {\rm Theorem~\ref{gehring} and~\ref{gehring_boundary}\footnote{Actually, we take the minimum between the higher integrability exponent $\chi_1$ arising in both the aforementioned theorems. For the sake of simplicity, we keep the same symbol.
}
}. Then for every $S>1$ and any $\lambda>0$
 we have the following:
if for some cylinder $\qq_\rho$, with $\rho<\min\{\rho_0,R_0\}/12$, $\rho_0$ as in  {\rm Definition~\ref{def_p_thick}} and $R_0$ the radius of a fixed ball $B_{R_0}$, it holds
\begin{equation*}
\left|\left\{ (x,t)\in \qq_\rho : M^*(\mathds{1}_{\Omega_T}|Du|)(x,t) > AS\lambda\right\}\right| \, \geq \, S^{-2\chi_1}|\qq_\rho|\,,
\end{equation*}
then there exists $\varepsilon=\varepsilon(S,\chi_1)$ such that
$$
\qq_\rho \subset \left\{ M^*(\mathds{1}_{\Omega_T}|Du|) > \lambda \right\} \cup \left\{ M^*_1(\mathds{1}_{\Omega_T} f)>\varepsilon\lambda\right\},
$$
where we denoted by
$M^*:=M^*_{0,\qq_{R_0}}$ and $M^*_1:=M^*_{1,\qq_{R_0}}$.
\end{lemma}
\begin{proof}
First, we notice that the statement can be restated as follows:
\vspace{1mm}

{\it Assume that for some cylinder $\qq_\rho$ with $\rho<\min\{\rho_0,R_0\}/12$ we have
\begin{equation}\label{hp}
\qq_\rho \cap \left\{ M^*(\mathds{1}_{\Omega_T}(|Du|) \leq \lambda \right\} \cap \left \{ M^*_1(\mathds{1}_{\Omega_T}f) \leq \varepsilon\lambda\right\} 
\neq \emptyset,
\end{equation}
then
\begin{equation}\label{eq_33}
\left|\left\{ (x,t)\in \qq_\rho : M^*(\mathds{1}_{\Omega_T}|Du|)(x,t)>AS\lambda\right\}\right| 
< S^{-2\chi_1}|\qq_\rho|.
\end{equation}
}
\vspace{1mm}

By the hypothesis \eqref{hp} there exists $(y_0,s_0)\in \mathcal{Q}_\rho\equiv \mathcal{Q}_\rho(y_1,s_1)$ such that, for any $0<r\leq R_0$, with $(y_0,s_0)\in \mathcal{Q}_r$,
\begin{equation}\label{conseq_hp}
\dashint_{\mathcal{Q}_r}\mathds{1}_{\Omega_T}|Du|\,{\rm d}y\,{\rm d}s\leq \lambda
\end{equation}
and
\begin{equation}\label{conseq_hp2}
r\,\dashint_{\mathcal{Q}_r}\mathds{1}_{\Omega_T}|f|\,{\rm d}y\,{\rm d}s\leq \varepsilon\lambda,
\end{equation}
reminding the definition of maximal operator.
\vspace{1mm}

We want to prove that, for any $(x,t)\in \mathcal{Q}_\rho(y_1,s_1)$, the following inequality holds
\begin{equation}\label{claim}
M^*(\mathds{1}_{\Omega_T}|Du|)(x,t)\leq\max\Big\{M^{*}(\mathds{1}_{\mathcal{Q}_{2\rho}(y_1,s_1)\cap\Omega_T}|Du|)(x,t), \, 3^N\lambda\Big\}.
\end{equation}
For this, we consider a generic cylinder $\mathcal{Q}_{r}$ such that $(x,t)\in \mathcal{Q}_{r}\subseteq \mathcal{Q}_{R_0}$; for any $r\leq \rho$, $\mathcal{Q}_{r}\subset \mathcal{Q}_{2\rho}(y_1,s_1)$, and so we have
$$
\dashint_{\mathcal{Q}_r}\mathds{1}_{\Omega_T}|Du|\,{\rm d}y\,{\rm d}s\leq M^{*}(\mathds{1}_{\mathcal{Q}_{2\rho}(y_1,s_1)\cap\Omega_T}|Du|)(x,t).
$$
On the other hand, for any $\rho<r\leq R_0$ we plainly have, enlarging, if necessary, the cylinder $\mathcal{Q}_{r}$, and using \eqref{conseq_hp} 
\begin{equation*}
\dashint_{\mathcal{Q}_r}\mathds{1}_{\Omega_T}|Du|\,{\rm d}y\,{\rm d}s\leq  3^N\lambda.
\end{equation*}
Thus, the inequality in~\eqref{claim} is proven. 
\vspace{1mm}

Now, let $\mathcal{Q}_{4\rho}\equiv\mathcal{Q}_{4\rho}(y_1,s_1) \supset \mathcal{Q}_{\rho}$ and distinguish the following three cases: $\mathcal{Q}_{4\rho}\subset \mathds{R}^{n+1}\setminus \Omega_T$, $\mathcal{Q}_{4\rho}\cap \,\partial_{{\rm par}} \Omega_T\neq \emptyset$ and $\mathcal{Q}_{4\rho}\subset \Omega_T$.
\vspace{0.5mm}

In the first case, the inequality in~\eqref{eq_33} does hold provided that $A\geq 3^N$, by taking into account~\eqref{claim}.
\vspace{0.5mm}

In order to deal with the second case, that is where the cylinder $\mathcal{Q}_{4\rho}$ intersects $\partial_{{\rm par}}\Omega_T$, we extend to the parabolic setting the strategy in Proposition~3.2 in~\cite{Phu12}.
Precisely, take any $(y_2,s_2) \in \partial_{{\rm par}}\Om_T$ such that 
$
{\rm dist}_{\rm par}\big((y_1,s_1), \partial_{{\rm par}}\Om_T\big):=|y_1-y_2|+|s_1-s_2|^{\frac 12}
$
and  
consider the unique solution  $w$ to the following homogeneous Cauchy-Dirichlet problem
\begin{equation*}
\begin{cases}
w_t - \text{\rm div}\, a(x,t,Dw) = 0 & \text{in}\ \mathcal{Q}_{12\rho}(y_2,s_2)\cap\Om_T \\
w=u & \text{on}\ \partial_{\rm par}(\mathcal{Q}_{12\rho}(y_2,s_2)\cap\Om_T) .
\end{cases}
\end{equation*}
Also, we extend by zero the function $u$ to $\R^{n+1}\setminus\Om_T$ and by $u$ the function $w$ to $\R^{n+1}\setminus(\mathcal{Q}_{12\rho}(y_2,s_2)\cap\Om_T) $.

Now, we are in a position to estimate the measure of the investigated level sets by using~\eqref{eq_8star} twice with $\gamma=2\chi_1$ and $\gamma=1$. It follows 
\begin{eqnarray}\label{eq_20star}
&&
\left|\left\{(x,t)\in \qq_\rho : M^{*}(\mathds{1}_{\mathcal{Q}_{2\rho}(y_1,s_1)\cap\Omega_T}|Du|) (x,t)> AS\lambda\right\}\right| \nonumber \\[1.5ex]
&& \qquad \qquad \leq \left|\left\{(x,t)\in \qq_\rho : M^{*}(\mathds{1}_{\mathcal{Q}_{2\rho}(y_1,s_1)\cap\Omega_T}|Dw|)(x,t) > AS\lambda/2\right\}\right| \nonumber\\[1ex]
&& \qquad \qquad \quad +\, \left|\left\{(x,t)\in \qq_\rho : M^{*}(\mathds{1}_{\mathcal{Q}_{2\rho}(y_1,s_1)\cap\Omega_T}|Du-Dw|) (x,t)> AS\lambda/2\right\}\right| \nonumber \\[1.5ex]
&& \qquad  \qquad \leq c(AS\lambda)^{-2\chi_1}\int_{\mathcal{Q}_{2\rho}(y_1,s_1)\cap\Omega_T}|Dw|^{2\chi_1}\,{\rm d}y\,{\rm d}s  \\[1ex]
&& \qquad \qquad \quad +\, c(AS\lambda)^{-1} \int_{\mathcal{Q}_{2\rho}(y_1,s_1)\cap\Omega_T}|Du-Dw|\,{\rm d}y\,{\rm d}s. \nonumber
\end{eqnarray}
By Theorem~\ref{gehring_boundary}, applied with $R=12\rho<\rho_0$, the first term in the  right-hand side of~\eqref{eq_20star} is bounded from above by
$$
c(AS\lambda)^{-2\chi_1}|\mathcal{Q}_{6\rho}|\left(\frac{1}{|\mathcal{Q}_{12\rho}|}\int_{\mathcal{Q}_{12\rho}(y_2,s_2)\cap\Om_T}|Dw|\,{\rm d}y\,\rm{d}s\right)^{2\chi_1},
$$
since $\mathcal{Q}_{2\rho}(y_1,s_1)\subset \mathcal{Q}_{6\rho}(y_2,s_2)$. 
Thus we get
\begin{eqnarray*}
&&
\left|\left\{(x,t)\in \qq_\rho : M^*(\mathds{1}_{\mathcal{Q}_{2\rho}(y_1,s_1)\cap\Omega_T}|Du|)(x,t) > AS\lambda\right\}\right| \nonumber \\[1ex]
&& \qquad  \qquad \leq c(AS\lambda)^{-2\chi_1}|\mathcal{Q}_{6\rho}|\left(\frac{1}{|\mathcal{Q}_{12\rho}|}\int_{\mathcal{Q}_{12\rho}(y_2,s_2)\cap\Omega_T}|Du|\,{\rm d}y\,{\rm d}s\right) ^{2\chi_1}\nonumber \\[1ex]
&& \qquad  \qquad \quad +\,  c(AS\lambda)^{-2\chi_1}|\mathcal{Q}_{6\rho}| \left(\frac{1}{|\mathcal{Q}_{12\rho}|}\int_{\mathcal{Q}_{12\rho}(y_2,s_2)\cap\Omega_T}|Du-Dw|\,{\rm d}y\,{\rm d}s\right)^{2\chi_1} \nonumber \\[1ex]
&& \qquad  \qquad \quad +\, c(AS\lambda)^{-1}\frac{ |\mathcal{Q}_{6\rho}|}{|\mathcal{Q}_{12\rho}|}\,\int_{\mathcal{Q}_{12\rho}(y_2,s_2)\cap\Omega_T}|Du-Dw|\,{\rm d}y\,{\rm d}s. \nonumber
\end{eqnarray*}

In view of the comparison estimates established in Lemma~\ref{comparison_boundary} (taking $R=12\rho<\rho_0$ there), together with~\eqref{conseq_hp} and \eqref{conseq_hp2}, since $(y_0,s_0)\in\mathcal{Q}_{\rho}(y_1,s_1)\subset \mathcal{Q}_{12\rho}(y_2,s_2)\subseteq \mathcal{Q}_{R_0}$, and choosing $\varepsilon=S^{-2\chi_1+1}$, we arrive at 
$$
\left|\left\{(x,t)\in \qq_\rho : M^*(\mathds{1}_{\mathcal{Q}_{2\rho}(y_1,s_1)\cap\Omega_T}|Du|)> AS\lambda\right\}\right|\leq c(A^{-2\chi_1}+A^{-1})S^{-2\chi_1}|\mathcal{Q}_\rho|.
$$
Now, it suffices to choose the constant~$A$ so that $A\geq 3^N$ and $2cA^{-1}
\,\leq\, 1/2$; i.~\!e., $A\geq \max\big\{3^N,\, 4c\big\}$. It follows 
$$
\left|\left\{(x,t)\in \qq_\rho : M^*(\mathds{1}_{\mathcal{Q}_{2\rho}(y_1,s_1)\cap\Omega_T}|Du|)(x,t) > AS\lambda\right\}\right|\leq \frac 1 2 S^{-2\chi_1}|\mathcal{Q}_\rho|,
$$
which, again in view of \eqref{claim}, yields \eqref{eq_33}.

\vspace{0.5mm}
Finally, when $\mathcal{Q}_{4\rho}\subset \Omega_T$, the inequality in~\eqref{eq_33} plainly follows by Lemma~6.2 in~\cite{BH12b}, hence using the local interior results given by Theorem~\ref{gehring} and Lemma~\ref{comparison}.

\end{proof}

\vspace{1.5mm}

Before starting with the proof of Theorem~\ref{thm_du}, we want to emphasize that most of the differences between the problem we are dealing with and the analog in the elliptic case analyzed in~\cite{Phu12} have arisen in the previous section. For this, the general strategy of the proof follows that of the proof of Theorem~6.1 in~\cite{BH12b} (see, also, \cite[Theorem 4.1]{DP12b}.
We prefer to give some details for the reader's convenience, and we sketch the proof in three steps.

\noindent
\\ {\it Proof of Theorem~\ref{thm_du}}. \ We denote by $\qq_0\equiv B_{R_0}(x_0)\times (t_0-R_0^2,t_0)$ a fixed cylinder in $\R^{n+1}$ that contains $\Omega_T$, with its basis $B_0\equiv B_{R_0}$ being a ball of radius $R_0 \leq 2 \text{diam}(\Om)$. Note that since $\mathcal{Q}_0 \supset \Omega_T$ we can assume $R_0\geq {\rm diam}(\Omega)/2$. Now we consider the following maximal operators
\begin{equation*}
M^*(\, \cdot\,):=M^*_{0,\qq_0}( \mathds{1}_{\Omega_T}\, \cdot\,)
\quad \text{and} \quad M^*_1(\, \cdot\,):=M^*_{1,\qq_0}( \mathds{1}_{\Omega_T}\, \cdot\,).
\end{equation*}

\noindent
\\ {\it Step 1 - Application of Calder\'on-Zygmund-Krylov-Safonov covering type lemma.}
We  consider the following sets $\mathcal{X}$ and $\mathcal{Y}$ defined
\begin{eqnarray*}
\mathcal{X}& := & \Big\{(x,t)\in \mathcal{Q}_0 : M^*(|Du|)(x,t)>(AS)^{k+1}\lambda_1 
\Big\},
\\
\\
\mathcal{Y}  & := & \Big\{(x,t)\in \mathcal{Q}_0 : M^*(|Du|)(x,t)>(AS)^{k}\lambda_1 \\
&& \qquad \qquad  \qquad\qquad \qquad  \qquad \qquad \text{and} \ M_1^*(f)(x,t)>\varepsilon(AS)^k\lambda_1\Big\},
\end{eqnarray*} 
with
\begin{equation}\label{def.lambda0}
 \lambda_1:=  \bar{c}  \,S^{2\chi_1}\,\frac{1}{|\mathcal{Q}_{\bar{R}}|}\int_{\mathcal{Q}_0\cap\, \Omega_T}|Du|\,{\rm d}x\,{\rm d}t,
\end{equation}
where $S>1$, $A$, $\chi_1$, $\varepsilon$ are as in Lemma~\ref{lem_37}. We want to apply Lemma~\ref{lem_covering} with $\bar{R}=\min\{\rho_0,R_0\}/12$. Firstly, we notice that by Lemma~\ref{lem_37} (choosing $\lambda=(AS)^k\lambda_1>0$ there) the assumption (ii) is satisfied with $\delta=S^{-2\chi_1}$. Thus, it remains to prove that $|\mathcal{X}|<S^{-2\chi_1}|\mathcal{Q}_{\bar{R}}|$. For this, it suffices to use the boundedness of the maximal operators in Marcinkiewicz spaces together with the choice of $\lambda_1$ in~\eqref{def.lambda0}. Indeed,
\begin{eqnarray*}
&&\left|\Big\{(x,t)\in \mathcal{Q}_0 : M^*(|Du|)(x,t)>(AS)^{k+1}\lambda_1 
\Big\}\right| \\
&& \qquad \qquad \quad\,\,\,\,
 <  \! \frac{\bar{c}}{(AS)^{k+1}\lambda_1}\int_{\mathcal{Q}_0\cap\,\Omega_T}|Du|\,{\rm d}x\,
{\rm d}t\\
&& \qquad \qquad \quad\,\,\,\, < S^{-2\chi_1}|\mathcal{Q}_{\bar{R}}|,
\end{eqnarray*} 
where we also used the fact that $A, S >1$.
Hence, the application of Lemma~\ref{lem_covering} and the definitions of $\mathcal{X}$ and $\mathcal{Y}$ yield
\begin{eqnarray*}
&& \left|\left\{(x,t)\in \qq_0 : M^*(|Du|)(x,t) > (AS)^{k+1}\lambda_1\right\}\right| \\
&& \qquad \qquad \qquad \leq \, c_3 S^{-2\chi}\left|\left\{(x,t)\in \qq_0 : M^*(|Du|)(x,t) > (AS)^k\lambda_1\right\}\right| \\
&& \qquad\qquad \qquad \quad \, + c_3\,\left|\left\{(x,t)\in\qq_0 : M^*_1(f)(x,t) > \varepsilon(AS)^k\lambda_1\right\}\right|,
\end{eqnarray*}
for every $k \in \mathbb{N}$; with $c_3$ depending only on $n,\rho_0$, and diam$(\Om)$.

\vspace{1mm}

Now, by taking into account a multiplication by a factor $(AS)^{\frac{(k+1)N\gamma}{N-\gamma}}$, from the previous estimate we deduce
\begin{eqnarray}\label{tesi.prop.bis}
&& (AS)^{\frac{(k+1)N\gamma}{N-\gamma}}\lambda_1^{\frac{N\gamma}{N-\gamma}}\mu_1((AS)^{k+1}\lambda_1) \nonumber \\[1ex]
 &&  \qquad  \qquad  \qquad 
 \leq
c_3 (AS)^{\frac{N\gamma k}{N-\gamma}} A^{\frac{N\gamma}{N-\gamma}}
 S^{\frac{N\gamma}{N-\gamma}-2\chi}\lambda_1^{\frac{N\gamma}{N-\gamma}}\mu_1((AS)^{k}\lambda_1) \\
 && \qquad  \qquad  \qquad \quad
 +\,c_3(AS)^{\frac{N\gamma k}{N-\gamma}}\left(\frac{AS}{\varepsilon}\right)^{\frac{N\gamma}{N-\gamma}}(\lambda_1\varepsilon)^{\frac{N\gamma}{N-\gamma}}\mu_2(\varepsilon(AS)^{k}\lambda_1), \nonumber
\end{eqnarray}
where, for any $K\geq 0$, we denoted by
\begin{equation}\label{def_mu1}
  \mu_1(K):=\big|\big\{(x,t)\in \mathcal{Q}_0 : M^*(|Du|)(x,t)>K\big\}\big|
\end{equation}
and
\begin{equation}\label{def_mu2}
 \mu_2(K):=\big|\big\{(x,t)\in \mathcal{Q}_0 : M_1^*(f)(x,t)>K\big\}\big|.
\end{equation}
Now, we note that, since $\chi_1>1$ and $\gamma\leq 2N/(N+2)$, 
 we have that the quantity
\begin{equation*}
d:=2\chi_1 - \frac{N\gamma}{N-\gamma} \, \geq \, 2(\chi_1-1) \, > \, 0
\end{equation*}
is positive. Therefore, we can choose
$$
S:= \left( 4A^{\frac{N\gamma}{N-\gamma}}\right)^{\frac{1}{d}}
$$
and, by computations, inequality~\eqref{tesi.prop.bis} provides the existence of  a constant $c\equiv c(n,L, \nu, \text{diam}(\Om),\rho_0)$ such that, for every $k\geq 0$,
\begin{align}\label{tesi.prop.fin}
(A&S)^{\frac{N\gamma(k+1)}{N-\gamma}}\lambda_1^{\frac{N\gamma}{N-\gamma}}\mu_1((AS)^{k+1}\lambda_1) \\[1ex]
&  \leq \frac{1}{4}(AS)^{\frac{N\gamma k}{N-\gamma}}\lambda_1^{\frac{N\gamma}{N-\gamma}}\mu_1((AS)^{k}\lambda_1)+ \, c(AS)^{\frac{N\gamma k}{N-\gamma}}(\varepsilon\lambda_1)^{\frac{N\gamma}{N-\gamma}}\mu_2((AS)^k\varepsilon\lambda_1). \notag
\end{align}

\noindent
\\ {\it Step 2 - Level sets estimates.} 
In order to establish some Lorentz spaces estimates on level sets we proceed as follows. We take $0<\beta<\infty$ and operate some computations to get
\begin{eqnarray*}
&& \int_0^\infty[\lambda^{\frac{N\gamma}{N-\gamma}}\mu_1(\lambda)]^{\frac{\beta(N-\gamma)}{N\gamma}}\frac{{\rm d}\lambda}{\lambda}\\
&&  \qquad  \leq\left(\frac 1\beta +\tilde{c}^{\beta}(AS)^{\beta}\log(AS)\right)\lambda_1^{\beta}|\mathcal{Q}_0|^{\frac{\beta(N-\gamma)}{N\gamma}}+ \tilde{c}^{\beta}(AS)^{\beta}\log(AS) J(\infty),
\end{eqnarray*}
where
$$
J(\infty):=\sum_{k=0}^\infty \left((AS)^{\frac{N\gamma k}{N-\gamma}}(\varepsilon\lambda_1)^{\frac{N\gamma}{N-\gamma}}\mu_2((AS)^k\varepsilon\lambda_1)\right)^{\frac{\beta(N-\gamma)}{N\gamma}}
$$
and the constant $ \tilde{c}>1$ is increasing in the variables $n$,  $L, \nu$ and decreasing in~$\beta$, such that $\tilde{c}\rightarrow \infty$ as $\beta\rightarrow 0$, while it remains bounded when $\beta$ is bounded away from zero.
Thus, from this inequality, we can plainly deduce
\begin{eqnarray}\label{L.est.4}
\int_0^\infty[\lambda^{\frac{N\gamma}{N-\gamma}}\mu_1(\lambda)]^{\frac{\beta(N-\gamma)}{N\gamma}}\frac{{\rm d}\lambda}{\lambda} 
 & \leq & \left(\frac 1\beta +2\tilde{c}^\beta(AS)^\beta\log(AS)\right)\lambda_1^\beta|\mathcal{Q}_0|^{\frac{\beta(N-\gamma)}{N\gamma}\notag}\\
&& \, +\, \tilde{c}^\beta(AS)^{2\beta}\int_0^\infty[\lambda^{\frac{N\gamma}{N-\gamma}}\mu_2(\lambda)]^{\frac{\beta(N-\gamma)}{N\gamma}}\frac{{\rm d}\lambda}{\lambda}. 
\end{eqnarray}
We recall the definitions of $\mu_1$ and $\mu_2$ given in~\eqref{def_mu1}-\eqref{def_mu2},
and, now, we choose\footnote{
We note that in the computations above we preferred  to keep the generic exponent $\beta$, since it could be chosen arbitrarily in order to obtain estimates also in other rearrangement and non rearrangement spaces.
} $\beta=q\in (0,\infty)$, so that, by the definition of parabolic maximal operators, we get\label{beta} 
\begin{equation}\label{eq_quasifatto}
\|M^*(|Du|)\|_{L\left(\frac{N\gamma}{N-\gamma}, q\right)(\mathcal{Q}_0)}
\, \leq \, \tilde{c}\,\lambda_1 |\mathcal{Q}_0|^{\frac{N-\gamma}{N\gamma}}+\tilde{c}\,\|M_1^*(f)\|_{L\left(\frac{N\gamma}{N-\gamma},q\right)(\mathcal{Q}_0)},
\end{equation}
up to relabeling the constant $\tilde{c}$, by keeping the same properties as before.
\vspace{1.5mm}

Now, thanks to a classical global estimate for the gradient of $u$
established in~\cite{BG89}, which we will use as in the form presented in~\cite[Lemma 5.1]{BH12} where the exact dependence on the $L^1$-norm appears, we can estimate $\lambda_1$ from above, as follows
$$
\lambda_1
\, \leq \, c(n,\rho_0, {\rm diam}(\Omega))|\mathcal{Q}_0|^{\frac 1N-1}\int_{\mathcal{Q}_0\cap\,\Om_T}|f|\,{\rm d}x\,{\rm d}t.
$$
Therefore, for any $(x,t)\in \mathcal{Q}_0$ we have 
$
M_1^*(f)(x,t)\geq{\lambda_1}/{c(n,\rho_0,{\rm diam}(\Om))},
$
and thus
$$
\lambda_1 |\mathcal{Q}_0|^{\frac{N-\gamma}{N\gamma}}\leq \|M_1^*(f)\|_{L\left(\frac{N\gamma}{N-\gamma}\right)(\mathcal{Q}_0)}.
$$
{F}inally, in view of the standard inequalities $|Du(x,t)| \leq M^*(|Du|)(x,t)$, for a.~\!e.  $(x,t)\in \qq_0$, and 
$\|M_1^*(f)\|_{L\left(\frac{N\gamma}{N-\gamma},q\right)(\mathcal{Q}_0)}\leq \|f\|_{L(\gamma,q)(\mathcal{Q}_0)}$, we  arrive at
\begin{equation}\label{L.Du.est.2}
\|Du\|_{L\left(\frac{N\gamma}{N-\gamma},q\right)(\Omega_T)}
\, \leq \, {c}\,\|f \|_{L(\gamma,q)(\mathcal{Q}_0)},
\end{equation}
where the constant $c$ depends only on $n,L,\nu,\gamma,q,\rho_0$ and $\text{diam}(\Om)$. Recalling that $\Om_T\subset \mathcal{Q}_0$ and $f=0$ in $\mathds{R}^{n+1}\setminus \Omega_T$ we plainly deduce the desired estimate in~\eqref{zina}.

\vspace{1.5mm}

Similarly, we can deal with the Marcinkiewicz case $q=\infty$ (see, e.g., Step 3 in the proof of \cite[Theorem 3]{BH12b}) and we arrive at
\begin{equation*}
\|Du\|_{\mathcal{M}^{\frac{N\gamma}{N-\gamma}}(\Omega_T)}
\, \leq \, 
{c}\,\|f\|_{\mathcal{M}^{\gamma}(\Omega_T)}.
\end{equation*}

\noindent
\\ {\it Step 3 - Conclusion of the proof.} We recall that we proved the estimate in~\eqref{L.Du.est.2} for the approximating solutions $u\equiv u_k$ to problem~\eqref{problema_approx} with $f\equiv f_k=T_k(f)$, where $T_k$ is the truncation operator defined by~\eqref{def_troncaf}. In order to conclude, it suffices to use the lower semicontinuity of the Lorentz norms together with the standard approximating arguments stated in Section~\ref{sec_solvability}.
Hence, we have proven that
\begin{equation}\label{L.M.Duk}
\|Du_k\|_{L\left({\frac{N\gamma}{N-\gamma}},q\right)(\Omega_T)} 
\, \leq 
{c}\,\|f_k\|_{L(\gamma,q)(\mathcal{Q}_0)},
\end{equation}
where the constant $c$ does not depend on~$k$. We also notice that in \eqref{L.M.Duk} we have used the fact that
$$
\|f_k\|_{L(\gamma,q)(\qq_0)}\leq \|f\|_{L(\gamma,q)(\qq_0)},
$$
since by the definition of $f_k$ it holds that $|f_k|\leq |f|$. Thus, in order to pass to the limit on $k\to \infty$ in~\eqref{L.M.Duk}, it suffices to use the lower semicontinuity of the Lorentz norms (recall remark on page~\pageref{lcs}) together with the approximating arguments stated in Section~\ref{sec_solvability} (recall, in particular,~\eqref{def_troncaf} there). \hfill$\square$

\vspace{1mm}

\begin{rem}
As in the classic case, Theorem~\ref{thm_du} fails for the borderline choice $\gamma=1$. 
Even in the elliptic case,  one has to impose some further $L\log L$ integrability on the datum $f$ in order to obtain the following implication
\begin{equation}\label{eq_jfa}
\displaystyle
f \in L^1(-T,0; L \log L(\Om)) \ \implies \ |Du| \in L^{\frac{N}{N-1}}(\Om_T);
\end{equation}
see Theorem~1.8 in \cite{BDGO97}. However, we could extend the potential approach in the proof of Theorem~\ref{thm_du} to the analysis of the borderline case in the parabolic Orlicz space $L\log L(\Om_T)$, to get, for any cylinder $\qq_\rho\subset\R^{n+1}$,  the following estimate
\begin{equation}\label{eq_b3}
\displaystyle
\| Du \|_{L^{\frac{N}{N-1}}(\qq_{\rho/2}\cap\Om_T)}
\, \leq  \, 
c\, \|f\|^{\frac{1}{N}}_{L^1(\qq_\rho\cap\Om_T)}\| f\|^{\frac{N-1}{N}}_{L\log L(\qq_\rho\cap\Om_T)}.
\end{equation}
We stress that when considering the preceding estimate in the whole $\Om_T$, the by-now classic result~\eqref{eq_jfa} in~\cite{BDGO97} is stronger,
since it does not require any regularity assumptions on the domain $\Omega$ and it does hold in the  
 weaker assumption of $f$ being in $L^1(-T,0; L \log 
L(\Om))$, despite ~\eqref{eq_b3} provides local estimates up to the boundary.
\end{rem}

\vspace{1mm}
\subsection{Integrability of $u$} 

This section is devoted to Lorentz space estimates for the solution $u$ to \eqref{problema}. We will use the same techniques used to prove the spatial regularity of the gradient. Hence, we will obtain an estimate on the level sets of the maximal operator associated to $u$, in terms of the level sets of a maximal operator of the datum $f$, up to a correction term; that is, the equivalent to \eqref{eq_quasifatto} for $u$.
This will allow to deduce the proof of Theorem~\ref{thm_u}, stated in the introduction, which is close to the one of Theorem~\ref{thm_du}, but also a higher regularity result is needed, as stated in the second part of Theorem~\ref{gehring_boundary} proven at the end of Section~\ref{sec_comparison}.

\vspace{2mm}

Keeping in mind the notation used in Theorem~\ref{thm_du}, we have the analog of Lemma~\ref{lem_37} for $u$. Again, we recall that in what is following we always deal with the approximating solutions $u\equiv u_k$ defined in Section~\ref{sec_solvability}, and with the datum $f\equiv f_k$.

\begin{lemma}\label{lem_37u}
Let $u\in C^0([-T,0]; L^2(\Om))\cap L^2(-T,0; W^{1,2}_0(\Om))$ be the solution to~\eqref{problema_approx}. Let $A>1$ be an absolute constant depending only on $n,L,\nu$, and let $\chi>1$ be the higher integrability exponent as in {\rm Theorems~\ref{gehring}} \!and {\,\rm \ref{gehring_boundary}}. 
Then for every $S>1$ and any $\lambda>0$ we have that exists $\varepsilon=\varepsilon(S,\chi)$ such that the following statement holds:
If for some cylinder $\qq_\rho$ with $\rho<\min\{\rho_0,R_0\}/12$ it holds
\begin{equation*}
\left|\left\{ (x,t)\in \qq_\rho : M^*(u) > AS\lambda\right\}\right| \geq S^{-2\chi}|\qq_\rho|
\end{equation*}
then
$$
\qq_\rho \subset \left\{ M^*(u) > \lambda \right\} \cup \left\{ M^*_2(f)>\varepsilon\lambda\right\},
$$ 
where we denoted by
$M^*(\cdot):=M^*_{0,\qq_{R_0}}(\mathds{1}_{\Omega_T}\cdot)$ and $M^*_2(\cdot):=M^*_{2,\qq_{R_0}}(\mathds{1}_{\Omega_T}\cdot)$.
\end{lemma}
\begin{proof}
The proof of this lemma is essentially the same as that of Lemma~\ref{lem_37}. It first suffices to substitute $M^*(|Du|)$ and $M^*_1(f)$ by
$M^*(u)$ and $M^*_2(f)$, respectively.
Then, we need to use the comparison estimates for $u$ stated in Lemma~\ref{comparison} and Lemma~\ref{comparison_boundary}, and the higher integrability, again for $u$, given in~\eqref{stima_gehringw} and \eqref{stima_gehring_boundary_u}. This will permit to introduce the parameter $\chi>1$ which can be chosen arbitrarily large; and this is another difference with respect to the estimates in Lemma~\ref{lem_37}, in which such a parameter (read $\chi_1>1$) was fixed.
\end{proof}

\vspace{2mm}

\noindent
\\ {\it Proof of Theorem~\ref{thm_u}.}
Since we closely follow the proof of Theorem~\ref{thm_du},  we prefer to keep the division of the proof in the same three steps.

\noindent
\\ {\it Step 1 - Application of Calder\'on-Zygmund-Krylov-Safonov covering type lemma}.
 We want to apply Lemma~\ref{lem_covering} to the sets  
\begin{eqnarray*}
\mathcal{X} & := &\Big\{(x,t)\in \mathcal{Q}_0 : M^*(u)(x,t)>(AS)^{k+1}\lambda_0 
\Big\},
\\
\\
\mathcal{Y}& := &\Big\{(x,t)\in \mathcal{Q}_0 : M^*(u)(x,t)>(AS)^{k}\lambda_0 \\
&& \qquad \qquad  \qquad\qquad \qquad  \qquad \qquad \text{and} \ M_2^*(f)(x,t)>\varepsilon(AS)^k\lambda_0\Big\},
\end{eqnarray*}
\begin{equation*}
\text{with} \, \ \lambda_0:=  \bar{c} \,S^{2\chi}\frac{1}{|\mathcal{Q}_{\bar{R}}|}\int_{\mathcal{Q}_0\cap\,\Omega_T}|u|\,{\rm d}x\,{\rm d}t\quad \text{and}\quad \bar{R}=\frac{\min\{\rho_0,R_0\}}{12}.
\end{equation*}
In view of Lemma~\ref{lem_37u}, the assumption (ii) of Lemma~\ref{lem_covering} (choosing $\lambda=(AS)^k\lambda_0$ there) is satisfied with $\rho<\bar{R}=\min\{\rho_0,R_0\}/12$ and $\delta=S^{-2\chi}$. To prove (i), that is $|\mathcal{X}|<S^{-2\chi}|\mathcal{Q}_0|$, it suffices to use the boundedness of the maximal operators in Marcinkiewicz spaces together with the choice of $\lambda_0$ and the fact that $A, S >1$. It follows
\begin{eqnarray}\label{tesi.prop.finu}
&&  \nonumber(AS)^{\frac{N\gamma(k+1)}{N-2\gamma}}\lambda_0^{\frac{N\gamma}{N-2\gamma}}\mu_1((AS)^{k+1}\lambda_0) \nonumber\\[1ex]
&&    \qquad    \qquad  \leq (AS)^{\frac{N\gamma k}{N-2\gamma}}
A^{\frac{N\gamma}{N-2\gamma}}S^{\frac{N\gamma}{N-2\gamma}-2\chi_0}
\lambda_0^{\frac{N\gamma}{N-2\gamma}}\mu_1((AS)^{k}\lambda_0)\\
 &&    \qquad    \qquad  \quad + \,(AS/\varepsilon)^{\frac{N\gamma}{N-2\gamma}} (AS)^{\frac{N\gamma k}{N-\gamma}}(\varepsilon\lambda_0)^{\frac{N\gamma}{N-\gamma}}\mu_2((AS)^k\varepsilon\lambda_0), \ \ \text{for every} \ k\geq 0, \nonumber
\end{eqnarray}
where, for any $K>0$,
 we denoted by
\begin{equation*}
  \mu_1(K):=\big|\big\{(x,t)\in \mathcal{Q}_0 : M^*(u)(x,t)>K\big\}\big|
\end{equation*}
and
\begin{equation*}
 \mu_2(K):=\big|\big\{(x,t)\in \mathcal{Q}_0 : M_2^*(f)(x,t)>K\big\}\big|.
\end{equation*}
At this level, we can take advantage of the possibility to choose $\chi$ large enough to satisfy
$$
d:= 2\chi - \frac{N\gamma}{N-2\gamma}>0,
$$
and, using the $S:=\left(4 A^{\frac{N\gamma}{N-2\gamma}}\right)^{\frac{1}{d}}$, it follows $A^{\frac{N\gamma}{N-2\gamma}} S^{\frac{N\gamma}{N-2\gamma}-2\chi} \leq 1/4$, so that the estimate in~\eqref{tesi.prop.finu} becomes
\begin{eqnarray*}
\nonumber(AS)^{\frac{N\gamma(k+1)}{N-2\gamma}}\lambda_0^{\frac{N\gamma}{N-2\gamma}}\mu_1((AS)^{k+1}\lambda_0) 
&  \leq & \frac{1}{4} (AS)^{\frac{N\gamma k}{N-2\gamma}}
\lambda_0^{\frac{N\gamma}{N-2\gamma}}\mu_1((AS)^{k}\lambda_0)\\[0.5ex]
 &&  \,+ \, c 
 (AS)^{\frac{N\gamma k}{N-\gamma}}(\varepsilon\lambda_0)^{\frac{N\gamma}{N-\gamma}}\mu_2((AS)^k\varepsilon\lambda_0), \nonumber
\end{eqnarray*}
for every  $k\geq 0$.

\noindent
\\ {\it Step 2 - Level sets estimates.}   Arguing as in 
 Step 2 in the proof of Theorem~\ref{thm_du} we obtain the following Lorentz  estimates for the solution $u$,
$$
\|u\|_{L\left(\frac{N\gamma}{N-2\gamma},q\right)(\Omega_T)}\leq {c}\,\|f \|_{L(\gamma,q)(\mathcal{Q}_0)},
$$
for any $0<q<\infty$. Similarly, in the case of Marcinkiewicz space, that is $q=\infty$, we arrive at
$$
\|u\|_{\mathcal{M}^{\frac{N\gamma}{N-2\gamma}}(\Omega_T)}
\, \leq \, {c}\,\|f\|_{\mathcal{M}^{\gamma}(\mathcal{Q}_0)},
$$
Notice that at this stage we used the fact that $M^*_2$ has a higher regularizing effect with respect to that of $M^*_1$. Moreover since we are assuming  $f=0$ in $\mathds{R}^{n+1}\setminus \Omega_T$ and $\Om_T\subset \mathcal{Q}_0$ we obtain exactly the desired estimate.

\noindent
\\ {\it Step 3 - Conclusion of the proof.} We can now conclude the proof using standard approximation methods, again by means of the lower semicontinuity of the Lorentz norms.
\hfill $\square$

\vspace{3mm}


\vspace{2mm}

\begin{thebibliography}{39}
\frenchspacing

\bibitem{AM07} {E. Acerbi, G. Mingione}: Gradient estimates for a class of parabolic systems. {\it Duke Math. J.} {\bf 136}~(2007), no. 2, 285--320.

\bibitem{Ada75} {D.~R. Adams}: A note on Riesz potentials. {\it Duke Math. J.} {\bf 42}~(1975), 765--778.


\bibitem{Ark95} {A.~A. Arkhipova}: $L_p$-estimates of the gradients of solutions of initial/boundary-value problems for quasilinear parabolic systems. Differential and pseudodifferential operators. {\em J. Math. Sci.} {\bf 73 (6)}~(1995), 609--617. 


\bibitem{BH12}{ P. Baroni, J. Habermann}: Calder\'on-Zygmund estimates for parabolic measure data equations. {\it J.~Differential Equations} {\bf 252}~(2012), 412--447.

\bibitem{BH12b}{ P. Baroni, J. Habermann}: New gradient estimates for parabolic equations. {\it Houston J. Math.}~{\bf 38}~(2012), 855--914.

\bibitem{Bar13}{P. Baroni}: Marcinkiewicz estimates for degenerate parabolic equations with measure data. In preparation. 

\bibitem{BM97} {D. Blanchard, F. Murat}: Renormalised solutions of nonlinear parabolic problems with $L^1$ data: existence and uniqueness. 
{\it Proc. Roy. Soc. Edinburgh Sect. A} {\bf 127}~(1997), no. 6, 1137--1152. 

\bibitem{BDGO97}{L. Boccardo, A. Dall'Aglio, T. Gallou\"et, L. Orsina}: {Nonlinear parabolic equations with measure data}. {\it  J.~Funct. Anal.} {\bf 147} (1997), 237--258.

\bibitem{BG89}{ L. Boccardo, T. Gallou\"et}: {Nonlinear elliptic and parabolic equations involving measure data}. {\it  J.~Funct. Anal.} {\bf 87} (1989), 149--169.


\bibitem{BDM10}{  V. B\"ogelein, F. Duzaar, G. Mingione}: {The boundary regularity of non-linear parabolic systems II}. {\it  Ann. Inst. Henri Poincar\'e Anal. Non Lin\'eaire} {\bf 27} (2010), 145--200.


\bibitem{CP89} { L. Caffarelli, I. Peral}: On $W^{1,p}$ estimates for elliptic equations in divergence form. {\it Comm. Pure Appl. Math.} {\bf 51}~(1989), 1--21.

\bibitem{DallA}{A. Dall'Aglio}: Approximated solutions of equations with $L^1$ data. Application to the $H$-convergence of quasi-linear parabolic equations. {\it Ann. Mat. Pura Appl. (IV)}, {\bf 170}~(1996), 207--240.

\bibitem{DP12} {A. Di Castro, G. Palatucci}: Measure data problems, lower order terms and interpolation effects. {\it Ann. Mat. Pura Appl.}, \href{http://dx.doi.org/10.1007/s10231-012-0277-7}{DOI: 10.1007/s10231-012-0277-7}

\bibitem{DP12b} { A. Di Castro, G. Palatucci}: Nonlinear parabolic problems with lower order terms and related integral estimates. {\it Nonlinear Anal.}, {\bf 75}~(2012), 4177--4197.

\bibitem{DP12c} { A. Di Castro, G. Palatucci}: Fractional regularity for nonlinear elliptic problems with measure data. {\it J. Convex Anal.}, {\bf 20}~(2013), no. 3. 


\bibitem{DPP03} {J. Droniou, A. Porretta, A. Prignet}: Parabolic capacity and soft measures for nonlinear equations. {\it Potential Anal.} {\bf 19}~(2003), no. 2, 99--161. 

\bibitem{DM11} { F. Duzaar, G. Mingione}: Gradient estimates via non-linear potentials. {\it Amer. J. Math.}, {\bf 133}~(2011) 1093--1149.

\bibitem{GiaquintaStruwe:1982} {M. Giaquinta, M. Struwe}: On the partial regularity of weak solutions of nonlinear parabolic systems. {\em Math. Z.} \textbf{179} (1982), 437--451.

\bibitem{Giu03}{ E. Giusti}: {\it Direct Methods in the Calculus of Variations}. World Scientific Publishing Co., Inc., River Edge, NJ, 2003.

\bibitem{Gra04} {L. Grafakos}: {\it Classical and Modern Fourier Analysis}. Pearson Edu. Inc., Upper Saddle River, 2004.

\bibitem{HKM93} {J. Heinonen, T. Kilpel\"ainen, O. Martio}: {\it Nonlinear potential theory of degenerate elliptic equations}. Oxford University Press, Oxford, 1993. 


\bibitem{HW08} { F. H\'elein, J. Wood}: Harmonic maps, in: {\it Handbook of Global Analysis}, vol. {\bf 1213}, Elsevier Sci. B.V., Amsterdam, 2008,  417--491.

\bibitem{KK94} {T.~Kilpel\"ainen, P. Koskela}: Global integrability of the gradients of solutions to partial differential equations. {\it  Nonlinear Anal.} {\bf 23} (1994), no. 7, 899--909.

\bibitem{KL} {J.~Kinnunen, J.L.~Lewis}: Higher integrability for parabolic systems of {$p$}-laplacian type.
	{\em Duke Math. J.} {\bf 102} (2000), 253--271.
	
\bibitem{KM11} { T. Kuusi, G. Mingione}: Nonlinear potential estimates in parabolic problems. {\it Rend. Lincei, Mat. e Appl.} {\bf 22} (2011), 161--174.


\bibitem{KM12b} { T. Kuusi, G. Mingione}: Gradient regularity for nonlinear parabolic equations. {\it Ann. Scu. Norm. Sup. Cl. Sci. {\rm (}5\,{\rm)}}, to appear.

\bibitem{KM12c} {T. Kuusi, G. Mingione}: The Wolff gradient bound for degenerate parabolic equations. {\it J. Eur. Math. Soc. (JEMS)}, to appear.

\bibitem{Lew88}{J. L. Lewis}: {Uniformly fat sets}. {\it Trans. Amer. Math. Soc.} {\bf 308}~(1988), 177--196.

\bibitem{Lie96}{G. M. Lieberman}: {\it Second order parabolic differential equations}. World Scientific Press, River Edge, 1996.

\bibitem{Mik96} { P. Mikkonen}: {On the Wolff potential and quasilinear elliptic equations involving measures}. {\it Ann. Acad. Sci. Fenn., Ser AI, Math. Dissert.} {\bf 104}~(1996), 1--71.



\bibitem{Min07} {G. Mingione}: {The Calder\'on-Zygmund theory for elliptic problems with measure data}. {\it Ann. Sc. Norm. Super. Pisa Cl. Sci.} {\rm (}5{\,\rm)} {\bf 6}~(2007), 195--261.

\bibitem{Min10} {G. Mingione}: {Gradient estimates below the duality exponent}. {\it Math. Ann.} {\bf 346} (2010), 571--627.

\bibitem{Min11} {G. Mingione}: Gradient potential estimates. {\it J. Eur. Math. Soc. {\rm(}JEMS\,{\rm)}} {\bf 13}~(2011), 459--486.
  

\bibitem{Min11b} {G. Mingione}: {Nonlinear measure data problems}. {\it Milan J. Math.} {\bf 79}~(2011), 429--496.

\bibitem{Min12c}{G. Mingione (ed.)}: {\it Topics in modern regularity theory}. CRM Series {\bf 13}, Edizioni della Normale, Pisa, 2012. 

\bibitem{Nik11} {R. Nikkta}: {Regularity of solutions of linear second order elliptic and parabolic boundary value problems on Lipschitz domains}. {\it J. Differential Equations} {\bf 251}~(2011), 860--880.

\bibitem{Par09} { M. Parviainen}: {Global gradient estimates for degenerate parabolic equations in nonsmooth domains}. {\it Ann. Mat. Pura Appl.} {\rm (}4{\,\rm)} {\bf 188}~(2009), 333--358.

\bibitem{Pet08} {F. Petitta}: {Renormalized solutions of nonlinear parabolic equations with general measure
data}. {\it Ann. Mat. Pura Appl.} {\rm (}4{\,\rm)} {\bf 187}~(2008), 563--604.

\bibitem{PPP11} {F. Petitta, A. Ponce, A. Porretta}: {Diffuse measures and nonlinear parabolic equations}, {\it J. Evol. Equations} {\bf 11}~(2011), no. 4, 861--905.

\bibitem{Phu12} {N.~C. Phuc}: {Global integral gradient bounds for quasilinear equations below or near the natural exponent}. {\it Ark. Mat.}, to appear.

\bibitem{Riv07} {T. Rivi\`ere}: Conservation laws for conformally invariant variational problems. {\it Invent. Math.}~{\bf 168}~(2007), 1--22.


\end{thebibliography}
\end{document}